\documentclass[11pt]{article}
\usepackage{graphicx}
\usepackage{xcolor}
\usepackage[hyphens]{url}
\usepackage{hyperref}

\usepackage{cite}

\usepackage{amsmath, amsthm, amssymb, amsfonts}
\usepackage{mathrsfs}
\usepackage{stmaryrd}
\usepackage{tikz-cd}
\usetikzlibrary{arrows}
\usepackage{enumerate}
\usepackage{indentfirst}

\usepackage[a4paper]{geometry}
\definecolor{blue-url}{RGB}{0,0,100}
\definecolor{red-url}{RGB}{100,0,0}
\definecolor{green-url}{RGB}{0,100,0}

\hypersetup{
	colorlinks=true,
	linkcolor=blue-url,
	citecolor=green-url,
	urlcolor=red-url
}

\newtheorem{theorem}{Theorem}

\newtheorem{open[theorem]}{Problem}
\newtheorem{question}[theorem]{Question}
\newtheorem{lemma}[theorem]{Lemma}

\newtheorem{corollary}[theorem]{Corollary}
\newtheorem{proposition}[theorem]{Proposition}
\theoremstyle{definition}

\newtheorem{remark}[theorem]{Remark}%[section]

\makeatletter
\renewcommand\maketitle{
	\begin{center}
		\hrule height2pt
		\vskip10pt
		{\LARGE\bfseries \@title \par}
		\vskip10pt
		\hrule height.5pt
		\vskip5pt
		{\large\bfseries \@author \par}
		\vskip3pt
		\hrule height2pt
	\end{center}
	\vskip6pt
}
\makeatother

\usepackage{tikz}
\usetikzlibrary{mindmap, shadows}

\title{\Large On the structural behavior of images of polynomials}

\author{Tsiu-Kwen Lee, Tran Nam Son\footnote{Tran Nam Son is the corresponding author.}}

\begin{document}
	
	\maketitle
	
	\begin{abstract}
		\footnotesize

 The study of images of noncommutative polynomials on algebras has attracted considerable attention. We investigate polynomial images and the additive structures they generate in associative algebras, focusing on sums and products of values. Motivated by results on additive commutators, we show that finite sums of such products on a nonzero ideal must contains a nonzero ideal, with only minor exceptions. Consequently, for a simple algebra, the subring generated by the image of a noncentral polynomial coincides with the whole algebra, up to a small exceptional case. We further study representations of elements as sums of products of polynomial values, and examine products of additive commutators for matrices over division rings. To simplify multilinear polynomials, we introduce decomposable polynomials and show that, in many cases, their images equal the whole algebra. Finally, we consider polynomial commutators and prove that every noncommutative infinite simple algebra is generated by such elements, together with results on multiplicative commutators, including a complete description for real quaternions.
	\end{abstract}
	
%	\vspace{10px}
	
	\noindent{\textbf{Keywords: }{Noncommutative polynomial; Prime ring; Simple ring; Matrix algebra;  Commutator.
			
			\noindent{\textbf{Mathematics Subject Classification 2020 (MSC2020):}{16N60; 16S50; 16U99;~47B47.
					
					%\vspace{10px}}
				
				\section{Introduction}\label{sec1}
				
				Let $F$ be a field, and let $F\langle \mathcal{X} \rangle$ denote the free associative $F$-algebra generated by the set $\mathcal{X}:=\{x_1, x_2, \ldots\}$, that is, the associative $F$-algebra of noncommutative polynomials in the noncommuting variables $x_i$. For any associative $F$-algebra $\mathcal{A}$ and any polynomial $p = p(x_1, x_2, \ldots, x_m) \in F\langle \mathcal{X} \rangle$ with $m \ge 1$, we define $$p(\mathcal{A}) = \{\, p(a_1, a_2, \ldots, a_m) \mid a_1, a_2, \ldots, a_m \in \mathcal{A} \,\},$$
				and refer to this set as the \textit{image} of $p$ evaluated on $\mathcal{A}$. Note that we tacitly assume that $p$ has zero constant term whenever $A$ is not unital.  Throughout this paper, the term \textit{algebra} is understood to mean an associative algebra. Furthermore, when we say that a polynomial $p$ is \textit{not central-valued} on an algebra $\mathcal{A}$, we mean that the set $p(\mathcal{A})$ is not contained in the center of $\mathcal{A}$.
				
			Noncommutative polynomials constitute one of the central objects in noncommutative algebra, and the problem of describing their images has long been both a classical and an actively developing area of research. In the classical setting, this line of study is closely related to the theory of rings with polynomial identities \cite{Pa_Ka_48}. More recently, the subject has experienced a revival-initiated approximately fifteen years ago-through the work of A.~Kanel-Belov, S.~Malev, and L.~Rowen \cite{Pa_KaMaRo_12} in connection with the L'vov-Kaplansky conjecture. This conjecture asserts that if $F$ is an infinite field, $n \ge 2$, and $p \in F\langle \mathcal{X} \rangle$ is a multilinear polynomial of the form $$\displaystyle p = \sum_{\sigma \in S_m} \lambda_\sigma \, x_{\sigma(1)} x_{\sigma(2)} \cdots x_{\sigma(m)},$$
			where $m$ is a positive integer and each $\lambda_\sigma \in F$, then the image $p(\mathrm{M}_n(F))$ is always a vector space. For the most recent developments, we refer the reader to the survey \cite{Pa_Ka_20}.
			
			Since such images are invariant under scalar multiplication, the main difficulty lies in determining whether they are also closed under addition. In addition, increasing attention has been devoted to additive subgroups \cite{Pa_Chu_87,  Pa_Lee_23,  Pa_Lee_22}, ideals \cite[Section 3]{Pa_Lee_22}, and linear spans \cite{Pa_Bre_20,Pa_Bre_23,Pa_Bre_23_1} generated by images of polynomials. Motivated by this perspective, the present paper investigates the behavior of images of noncommutative polynomials by examining the algebraic substructures they generate, in particular the additive subgroup and the subring arising from these images.  This paper also revisits the behavior of subrings generated by additive commutators, as investigated in recent works \cite{Pa_Er_22, Pa_Ga_25}.
			
		Another motivation arises from \cite[Theorem~4.9]{Pa_PaSo_26}, which shows that, under suitable conditions, every matrix over a centrally finite algebraically closed division ring can be expressed as the difference of two multiplicative commutators from $p(\mathrm{M}_n(D))$. From a ring-theoretic perspective, this naturally prompts a more detailed investigation of the subring of $\mathrm{M}_n(D)$ generated by elements of the image $p(\mathrm{M}_n(D))$, and raises structural questions concerning the multiplicative behavior of such images.

			Concretely, our paper is systematically organized as follows.  Section~\ref{sec sub} is devoted to additive subgroups generated by images, motivated by \cite{Pa_Ga_25}, where it is shown that double products of additive commutators suffice for generation. We begin with Theorem~\ref{thm27}, which shows that in a prime ring, the set of all finite sums of products of two values of a polynomial on a nonzero ideal necessarily contains a nonzero ideal, with only minor exceptions. As a consequence, Corollary~\ref{cor21} establishes that if $R$ is a simple $F$-algebra and $p\in F\langle\mathcal{X}\rangle$ is not central-valued on $R$, then $R$ coincides with the subring generated by $p(R)$, again up to a small exceptional case.
		
		Furthermore, Theorem~\ref{thm31} shows that if $R$ is an algebra over a field $F$ with $|F|>2$ and $p\in F\langle\mathcal{X}\rangle$ is such that $p(R)$ is fully noncentral, then for every maximal ideal $M$ of $R$ one has
		$
		R=(p(R)^2)^+ + M,
		$
		where $(p(R)^2)^+$ denotes the additive subgroup generated by all products $ab$ with $a,b\in p(R)$. The remainder of Section~\ref{sec sub} is concerned with the possibility of expressing elements as sums of products of images, as well as with the question of whether a uniform bound exists on the number of summands required (see Theorems~\ref{thm76}, \ref{thm77}, \ref{thm80}, and~\ref{thm81}).

		Section~\ref{sec product} focuses on the problem of expressing elements as products of additive commutators. This section aims to address several aspects of the still difficult question of when a matrix over a division ring $D$ can be written as a product of two additive commutators. We consider, in particular, the case where the center of $D$ is finite (Theorem~\ref{q>n}), the existence of division rings in which every element can be expressed as such a product and whose center is a prescribed field (Theorem~\ref{exist}), and the corresponding result for matrices over these division rings (Theorem~\ref{mainthm}). We also examine related topics, including algebraicity of division rings (Theorem~\ref{algebraic}) and finitary matrices (Theorem~\ref{finitary1}).

		In an effort to reduce the complexity of multilinear polynomials, we turn to the idea of factorization. Specifically, we describe the images of so-called \textit{decomposable} multilinear polynomials, namely those that can be written as products of two multilinear polynomials in disjoint sets of variables. We show that, in certain settings, their images coincide with the entire algebra, namely, for finitary matrix algebras (Theorems~\ref{Vitas1} and \ref{Vitas2}) and for matrix algebras over division rings (Theorem~\ref{Waring1}).

		Section~\ref{sec po} is devoted to a nonlinear analogue of additive commutators, namely polynomial commutators, introduced by Laffey and West for matrices over fields \cite{Pa_LaWe_93}. Among other results, Corollary~\ref{cor23} shows that every noncommutative infinite simple algebra is generated, as a ring, by its polynomial commutators.
		
		Finally, motivated by \cite[Theorem~4.9]{Pa_PaSo_26}, we investigate when an element can be expressed as the difference of two elements from the subgroup generated by multiplicative commutators. In particular, for the division ring of real quaternions, we show that such elements are precisely those with norm at most $2$ (Theorem~\ref{qu}). We also study analogous questions for matrices over the real quaternions (Theorem~\ref{real}).
				
				\section{Additive subgroups generated by polynomials}\label{sec sub}
				
				Let $R$ be a ring.
				For $x, y\in R$, let $[x, y]:=xy-yx$, the additive commutator of $x$ and $y$.
				Let $X, Y$ be subsets of $R$. We denote by $X^+$ (resp. ${\overline X}$) the additive subgroup (resp. subring) of $R$ generated by $X$, and by
				$
				X\bullet Y:=\{xy\mid x\in X, y\in Y\}.
				$
				For additive subgroups $A, B$ of $R$, let $[A, B]$ (resp. $AB$) denote the additive subgroup of $R$ generated by the elements $[a, b]$ (resp. $ab$) with $a\in A$ and $b\in B$. Thus $(X\bullet Y)^+=X^+Y^+$ for subsets $X, Y$ of $R$, and by $X^{[n]}:=X\bullet\cdots\bullet X$ ($n$-copies).
				Given subsets $X_1,\ldots,X_n$ of $R$ where $n>1$, $(X_1X_2\cdots X_n)^+$ stands for the additive subgroup of $R$ generated by all elements $x_1x_2\cdots x_n$ with $x_i\in X_i$ for all $i$.
				
				A ring $R$ is called a \textit{prime ring} if, for $a, b\in R$, $aRb=0$ implies that either $a=0$ or $b=0$. Given a prime ring $R$, let $Q_s(R)$ denote the Martindale symmetric ring of quotients of $R$. Then $Q_s(R)$ is also a prime ring. The center, denoted by $C$, of $Q_s(R)$ is a field, which is called the \textit{extended centroid} of $R$. The notion of extended centroid plays a key role in the theory of prime rings. We refer the reader to the book \cite{Bei_96} for details. Let $S$ be a ring, $n$ a positive integer. We denote ${\rm M}_n(S)$ the $n$ by $n$ matrix ring over $S$. If $S$ is unital, let $e_{ij}$ denote the standard matrix unit with $1$ in the $(i,j)$-position and $0$ elsewhere. Also, $\mathrm{GF}(2)$ denotes the field with two elements.
				
				We are now ready to state the first main result of this section.
				
				\begin{theorem}\label{thm27}
					Let $R$ be a prime ring with extended centroid $C$, and let $I$ be a nonzero ideal of $R$. Suppose that $f\in C\langle \mathcal{X}\rangle$ is not central-valued on $RC$. Then $(f(I)^2)^+$ contains a nonzero ideal of $R$ except when $R\cong {\rm M}_2({\rm GF}(2))$ and
$$
(f(R)^2)^+=\left\{0, 1, \begin{pmatrix}
						1 & 1 \\
						1 & 0
					\end{pmatrix}, \begin{pmatrix}
						0 & 1 \\
						1 & 1
					\end{pmatrix}\right\}.
$$
				\end{theorem}
				
				To establish Theorem \ref{thm27}, we rely on the following auxiliary results.
				
				\begin{theorem}\label{thm4} {\rm\cite[Theorem 2]{Pa_Cha_01}}
					Let $R$ be a prime ring with extended centroid $C$, and let $I$ be a nonzero ideal of $R$. Suppose that $f\in C\langle \mathcal{X}\rangle$ is not central-valued on $RC$. Then there exists a nonzero ideal $M$ of $R$ such that
					$[M,R]\subseteq f(I)^+,$
					except when $R\cong {\rm M}_2({\rm GF}(2))$ and either
					$
					f(R)^+=\left\{0, \begin{pmatrix}
						0 & 1 \\
						1 & 0
					\end{pmatrix}, \begin{pmatrix}
						1& 1 \\
						0 & 1
					\end{pmatrix},  \begin{pmatrix}
						1& 0 \\
						1 & 1
					\end{pmatrix}\right\}$
or $f(R)^+=\left\{0, 1, \begin{pmatrix}
						1 & 1 \\
						1 & 0
					\end{pmatrix}, \begin{pmatrix}
						0 & 1 \\
						1 & 1
					\end{pmatrix}\right\}.
					$
				\end{theorem}
				
				By a \textit{Lie ideal}  of a ring $R$ we mean an additive subgroup $L$ of $R$ satisfying $[L, R]\subseteq L$. A Lie ideal $L$ of $R$ is called \textit{abelian} (resp. \textit{nonabelian})
				if $[L, L]=0$ (resp. $[L, L]\ne 0$).
				
				\begin{lemma}\label{lem1} {\rm\cite[Theorem 3.5]{Pa_Lee_25}}
					Let $L$ be a nonabelian Lie ideal of a prime ring $R$. Then $L^2$ contains a nonzero ideal of $R$.
				\end{lemma}
				
				With these ingredients in place, we can now prove Theorem \ref{thm27}.
				
				\begin{proof}[Proof of Theorem~\ref{thm27}]
					Note that $R$ is noncommutative. The proof is divided into two cases.
					
					\textit{Case 1}: $f(I)^+$ contains a proper Lie ideal of $R$ (that is, a Lie ideal of the form $[M,R]$ with $M$ a nonzero ideal of $R$).
Clearly, $M$ is also a noncommutative prime ring.
					In view of \cite[Lemma~1.5]{Bo_Her_69}, $\big[[M, M], [M, M]\big]\ne 0$ and hence $[M,R]$ is a nonabelian Lie ideal of $R$. Applying Lemma \ref{lem1}, we conclude that $[M,R]^2$ contains a nonzero ideal $K$ of $R$. Consequently,
					$K\subseteq (f(I)^2)^+$, as desired.
					
					\textit{Case 2}: $f(I)^+$ does not contain any proper Lie ideal of $R$. 	Taking Theorem \ref{thm4} into account, it follows that
$I=R\cong {\rm M}_2({\rm GF}(2))$ and either
$$
f(R)^+=\left\{0, \begin{pmatrix}
						0 & 1 \\
						1 & 0
					\end{pmatrix}, \begin{pmatrix}
						1& 1 \\
						0 & 1
					\end{pmatrix},  \begin{pmatrix}
						1& 0 \\
						1 & 1
					\end{pmatrix}\right\}\ \ \text{\rm or}\ \
f(R)^+=\left\{0, 1, \begin{pmatrix}
						1 & 1 \\
						1 & 0
					\end{pmatrix}, \begin{pmatrix}
						0 & 1 \\
						1 & 1
					\end{pmatrix}\right\}.
$$	
If
$f(R)^+=\left\{0, \begin{pmatrix}
						0 & 1 \\
						1 & 0
					\end{pmatrix}, \begin{pmatrix}
						1& 1 \\
						0 & 1
					\end{pmatrix},  \begin{pmatrix}
						1& 0 \\
						1 & 1
					\end{pmatrix}\right\}$,
then a straightforward verification shows that
$$
(f(R)^2)^+=\left\{0, 1, \begin{pmatrix}
						1 & 1 \\
						1 & 0
					\end{pmatrix}, \begin{pmatrix}
						0 & 1 \\
						1 & 1
					\end{pmatrix}\right\}.
$$	
For the second possibility, $f(R)^+=\left\{0, 1, \begin{pmatrix}
						1 & 1 \\
						1 & 0
					\end{pmatrix}, \begin{pmatrix}
						0 & 1 \\
						1 & 1
					\end{pmatrix}\right\}$ and hence, by a direct computation, $f(R)^+=(f(R)^2)^+.$
In either case, we have
$
(f(R)^2)^+=\left\{0, 1, \begin{pmatrix}
						1 & 1 \\
						1 & 0
					\end{pmatrix}, \begin{pmatrix}
						0 & 1 \\
						1 & 1
					\end{pmatrix}\right\}.
$
This completes the proof.
				\end{proof}
				
				Recall that a ring $R$ is called \textit{simple} if $R^2 \neq 0$ and the only ideals of $R$ are $0$ and itself. Note that if $R$ is a simple ring with extended centroid $C$, then $R=RC$ and hence $R$ is a $C$-algebra. Consequently, Theorem~\ref{thm27} immediately leads to the following corollaries.
				
				\begin{corollary}\label{cor21}
					Let $R$ be a simple ring with extended centroid $C$. Suppose that $f \in C\langle \mathcal{X}\rangle$ is not central-valued on $R$. Then
$R = (f(R)^2)^+$ except when $R \cong \text{\rm M}_2(\text{\rm GF}(2))$ and
					$$
					(f(R)^2)^+ =\left\{0, 1, \begin{pmatrix}
						1 & 1 \\
						1 & 0
					\end{pmatrix}, \begin{pmatrix}
						0 & 1 \\
						1 & 1
					\end{pmatrix}\right\}.
					$$
				\end{corollary}
				
				\begin{corollary}\label{cor22}
					Let $D$ be a division ring with center $F$, and let $n$ be a positive integer. Suppose that $f \in F\langle \mathcal{X} \rangle$ is not central-valued on $\mathrm{M}_n(D)$. Then $\mathrm{M}_n(D) = (f(\mathrm{M}_n(D))^2)^+,$
					except when $n=2$, $D$ has exactly two elements, and
					$$
					(f(\mathrm{M}_n(D))^2)^+
					=\left\{0, 1, \begin{pmatrix}
						1 & 1 \\
						1 & 0
					\end{pmatrix}, \begin{pmatrix}
						0 & 1 \\
						1 & 1
					\end{pmatrix}\right\}.
					$$
				\end{corollary}
				
				In addition, Theorem~\ref{thm27} leads to the following result, for which we first recall a useful notion. Let $R$ be a ring. An additive subgroup $A$ of $R$ is said to be \textit{fully noncentral} if $\widetilde{R}[A,R]\widetilde{R}=R$. Here $\widetilde{R}$ denotes the minimal unitization of $R$: it coincides with $R$ when $R$ is unital, while in the nonunital case it is given by $\widetilde{R}=\mathbb{Z}\times R$, equipped with coordinatewise addition and multiplication defined by $(m,x)(n,y)=(mn,\, my+nx+xy).$
				
				With this terminology in place, we obtain the following theorem.
				
				\begin{theorem}\label{thm31}
					Let $R$ be an algebra over a field $F$ with $|F|>2$. Suppose that $f\in F\langle \mathcal{X} \rangle$ is a polynomial such that $f(R)$ is fully noncentral. Then for every maximal ideal $M$ of $R$ one has
					$R = (f(R)^2)^+ + M.$
				\end{theorem}
				
				\begin{proof}
					Since $f(R)$ is fully noncentral, we have $\widetilde{R}[f(R),R]\widetilde{R} = R$, and in particular it follows that $R=R^2$. Let $M$ be a maximal ideal of $R$. Then $(R/M)^2 = R/M \neq 0$, and hence $R/M$ is a simple algebra over $F$. Moreover,
$
\widetilde{R/M}[f(R/M),R/M]\widetilde{R/M} = R/M,
$
implying that
$f(R/M)$ is not central-valued on $R/M$, and the assumption $|F|>2$ ensures that the exceptional situation from Theorem~\ref{thm27} does not arise. Taking Theorem~\ref{thm27} into account, it follows that $(f(R/M)^2)^+$ contains a nonzero ideal of $R/M$. Since $R/M$ is simple, this implies that $(f(R/M)^2)^+ = R/M$. Consequently, lifting this equality back to $R$ yields $R = (f(R)^2)^+ + M,$ as required.
				\end{proof}
				
				Continuing along the same line of investigation, we recall a result from \cite[Theorem~3.4]{Pa_Ga_25}. Let $R$ be a unital ring and $f(x)=x \in F\langle \mathcal{X} \rangle$. Then $f(R)$ is fully noncentral if and only if $R=\widetilde R[R, R]\widetilde R$, that is, $R$ is generated by its commutators as an ideal.
				Gardella and Thiel proved that if $R$ is a unital ring, which is generated by its commutators as an ideal, then
				there exists a positive integer $N$ such that every element $x \in R$ can be expressed as
				$$\displaystyle x = \sum_{i=1}^N[a_i, b_i][c_i, d_i]$$ for some $a_i, b_i, c_i, d_i\in R$ (see \cite[Theorem A]{Pa_Ga_25}).
				To encompass our results regarding rings, by a {\it $K$-algebra} $R$ we mean that  $K$ is a unital commutative ring, which is not necessarily a field. In this case, a ring is always an algebra over $\Bbb Z$, the ring of integers.
				Motivated by the theorem, it is natural to ask whether a similar phenomenon occurs in the setting considered here.
				
				\begin{question}\label{qu1}
					Let $R$ be a unital $K$-algebra,  which  is generated by its commutators as an ideal, and let $f \in K\langle \mathcal{X} \rangle$ be a polynomial. Suppose that $R=(f(R)^n)^+$, where $n>1$ is a positive integer.
					Does there exist a positive integer $N$ such that, for $x \in R$, $\displaystyle x =a_1+\cdots+a_N$
					for some elements $a_1, \ldots, a_N \in f(R)^{[n]}:=f(R)\bullet \cdots\bullet  f(R)$ ($n$-copies)?
				\end{question}
				
			Let $L$ and $K$ be Lie ideals of a ring $R$. Then both the product $KL$ and the commutator $[K,L]$ are Lie ideals of $R$. Moreover, the set $RL$ is an ideal of $R$. Indeed, one has
			\[
			RLR \subseteq R([L,R] + RL)=R[L,R] + R^2L \subseteq RL + RL = RL,
			\]
			which shows that $RL$ is closed under multiplication by elements of $R$ on both sides. We also recall the well-known fact that $[A^+, R] = [\overline{A}, R]$ for any subset $A \subseteq R$. To see this claim, note that for all $x,a,b \in R$,
			$
			[x,ab] = [xa,b] + [bx,a],
			$
			which implies that $$[R,ab] \subseteq [R,a] + [R,b].$$ By induction, for any $a_1,\dots,a_m \in A$,
			$\displaystyle
			[R, a_1 a_2 \cdots a_m] \subseteq \sum_{i=1}^m [R,a_i].
			$
			Hence $$[R, A^m] \subseteq [R, A^+],$$ and it follows that $[A^+, R] = [\overline{A}, R]$.
			
			We are now in a position to address Question~\ref{qu1} in the case where the polynomial $f$ is multilinear.
				
				\begin{theorem}\label{thm76}
					Let $R$ be a unital $K$-algebra, and let $f \in K\langle \mathcal{X} \rangle$ be a multinear polynomial.
					Suppose that $R$ is generated by its commutators as an ideal and $R=(f(R)^n)^+$, where $n>1$ is a positive integer. Then there exists $N\in \Bbb N$ such that, for $x\in R$,
					$x=a_1+\cdots+a_N$ for some $a_i\in f(R)^{[n]}$.
				\end{theorem}
				
				\begin{proof}
					Since $R$ is a unital algebra generated by its commutators as an ideal, it follows from \cite[Theorem 3.10]{calugareanu2026} that $R=[R, R]^2$.
					In particular, $R=R^2$ and so $R=R^s$ for any positive integer $s$. We claim that $[R, R]^s=R$ for $s>1$.
					Indeed, it is clear if $s$ is even. Assume that $n=2k+1$ for some positive integer $k$. Then
					$$
					[R, R]^{2k+1}=[R, R]^{2k}[R, R]=R[R, R]\supseteq [R, R]^2=R,
					$$
					as desired.
					
					Let $t$ be the degree of the multilinear polynomial $f(x_1,\ldots,x_t)$.
					Note that, for elements $x, a_1,\ldots,a_t\in R$, we have
					\begin{equation}\label{eq:1}
						[x, f(a_1,\ldots,a_t)]=\sum_{i=1}^t f(a_1,\ldots,[x, a_i],\ldots,a_t).
					\end{equation}
					This implies that $f(R)^+$ is a Lie ideal of $R$. Moreover, the element $[x, f(a_1,\ldots,a_t)]$ is a sum of $t$ elements in $f(R)$.
					We further claim the following facts:\ Let $x\in R$ and $y_1,\ldots,y_n\in f(R)$.
					\begin{itemize}
						\item Claim 1:\ $[x, y_1\cdots y_n]$ is a sum of $tn$ elements of $f(R)^{[n]}$;\vskip4pt
						\item Claim 2:\ $[x, y_1\cdots y_n]$ is a sum of $tn$ elements of $f(R)$;\vskip4pt
						\item Claim 3:\ $R=[R, R]^{n-1}[[R, R]^{n-1}, R]$;\vskip4pt
						\item Claim 4:\ $[R, R]\subseteq f(R)^+$ and $(f(R)^{n-1})^+[(f(R)^{n-1})^+, (f(R)^n)^+]=R$.\vskip4pt
					\end{itemize}
					
					For Claim 1, we have
					$$
					[x, y_1\cdots y_n]=\sum_{i=1}^n( y_1\cdots[x, y_i]\cdots y_n).
					$$
					Since each $[x, y_i]\in [x, f(R)]$ is a sum of $t$ elements of $f(R)$, $[x, y_1\cdots y_n]$ is a sum of $tn$ elements of $f(R)^{[n]}$, as claimed.
					
					For Claim 2, we have
					$$
					[x, y_1\cdots y_n]=-[y_n, xy_1\cdots y_{n-1}]-[ y_{n-1}, y_nxy_1\cdots y_{n-2}]-\cdots-[ y_1, y_2\cdots y_nx].
					$$
					By Eq.(\ref{eq:1}), each of $-[y_n, xy_1\cdots y_{n-1}], -[ y_{n-1}, y_nxy_1\cdots y_{n-2}]$ and $-[ y_1, y_2\cdots y_nx]$ is a sum of $t$ elements of $f(R)$.
					Thus $[x, y_1\cdots y_n]$ is a sum of $tn$ elements of $f(R)$, as claimed.
					
					For Claim 3, since $n>1$, we have $[R, R]^{2(n-1)}=R$. Thus
					\begin{eqnarray*}
						R&=&[R, R]^{n-1}[R, R]\\
						&=&[R, R]^{n-1}\big[[R, R]^{2(n-1)}, R\big]\\
						&\subseteq& [R, R]^{n-1}\big[\overline {[R, R]^{n-1}}, R]\\
						&=& [R, R]^{n-1}\big[[R, R]^{n-1}, R].
					\end{eqnarray*}
					This proves Claim 3.
					For Claim 4, we have
					$$
					[R, R]=[R, (f(R)^n)^+]\subseteq [R, \overline{f(R)}]=[R, f(R)^+]\subseteq f(R)^+.
					$$
					That is, $[R, R]\subseteq f(R)^+$. Then, by Claim 3, we have
					\begin{eqnarray*}
						R&=& [R, R]^{n-1}[[R, R]^{n-1}, R] \\
						& \subseteq & (f(R)^{n-1})^+[(f(R)^{n-1})^+, (f(R)^n)^+],
					\end{eqnarray*}
					proving Claim 4. Taking Claim 4 into account, there exist finitely many $w_i, z_i\in f(R)^{[n-1]}$ and $v_i\in f(R)^{[n]}$, $1\leq i\leq m$, such that
					$
					1 = \sum_{i=1}^mw_i[z_i, v_i].
					$
					Let $x\in R$. Then
					\begin{eqnarray*}
						x &=& \sum_{i=1}^mxw_i[z_i, v_i] \\
						&=& \sum_{i=1}^m\big([xw_iz_i, v_i] - [xw_i, v_i]z_i ). \\
					\end{eqnarray*}
					By Claim 1, each of $[xw_iz_i, v_i]\in [xw_iz_i, f(R)^{[n]}]$, $i=1,\ldots,m$, is a sum of $tn$ elements of $f(R)^{[n]}$. Also,  by Claim 2, each of $ [xw_i, v_i]$, $i=1,\ldots,m$, is a sum of $tn$ elements of $f(R)$. This implies that each of $ [xw_i, v_i]z_i$, $i=1,\ldots,m$, is a sum of $tn$ elements of $f(R)\bullet f(R)^{[n-1]}$ (i.e., $f(R)^{[n]}$).
					Thus $x$ is a sum of $N$ elements of $f(R)^{[n]}$, where $N=2mtn$.
					This completes the proof.
				\end{proof}
				
				We note that if $f$ is not multilinear, the answer to Question \ref{qu1} remains unknown. The following question seems both interesting and important for generalizing Theorem \ref{thm76} to its full generality.
				
				\begin{question}\label{qu2}
					Let $R$ be a unital $K$-algebra. Characterize polynomials $f (x_1,\ldots,x_t)\in K\langle \mathcal{X} \rangle$ satisfying
					(i) $f(R)^+$ is a Lie ideal of $R$, and (ii) there exists a positive integer $N$ such that, for $x, y_1,\ldots,y_t\in R$,
					$[x, f(y_1,\ldots,y_t)] $ is a sum of $N$ elements of $f(R)$.
				\end{question}

				We remark that every multilinear polynomial $f(x_1,\ldots,x_t) \in K\langle \mathcal{X} \rangle$ satisfies conditions (i) and (ii) in Question~\ref{qu2}. When $K$ is an infinite field, it was shown by Bre\v{s}ar and Klep that the linear span of $f(R)$, viewed as a subspace of the $K$-algebra $R$, forms a Lie ideal of $R$ (see \cite[Theorem 2.3]{bresar2009}).
				
				We introduce the following notation. Let $[R,R]^{[1]} := \{[a,b] \mid a,b \in R\}$, and for each integer $n > 1$, define $$
				[R, R]^{[n]}:=\{[a_1, b_1][a_2, b_2]\cdots[a_n, b_n]\mid a_i, b_i\in R\ \forall i\}
				$$ to be the set of all products of $n$ commutators of the form $[a_1,b_1][a_2,b_2]\cdots[a_n,b_n]$, where $a_i,b_i \in R$ for all $i$. In particular, the additive subgroup $[R,R]^n$ is generated by $[R,R]^{[n]}$. Note that
				\[
				[R,R]^{[n-1]} \cdot [R,R]^{[n-1]} = [R,R]^{[2n-2]} \subseteq [R,R]^{2n-2}.
				\]
				
				As a consequence of Theorem~\ref{thm76}, we obtain a generalization of \cite[Theorem A]{Pa_Ga_25}; the special case $n=2$ is established below.
				
				\begin{theorem}\label{thm77}
					Let $R$ be a unital ring, which is generated by commutators as an ideal, and $n>1$  a positive integer. Then $R=[R, R]^n$ and there exists $N\in \Bbb N$ such that, for $x\in R$,
					$x=a_1+\cdots+a_N$ for some $a_i\in [R, R]^{[n]}$.
				\end{theorem}
				
				\begin{proof}
					Since $R$ is a unital ring generated by its commutators as an ideal, it follows from \cite[Theorem 3.10]{calugareanu2026} that $R=[R, R]^2$.
					As given in the proof of Theorem \ref{thm76}, it follows that $R=[R, R]^n$. That is, $R=(f(R)^n)^+$, where $f(x, y):=xy-yx\in \Bbb Z\langle x, y\rangle$.
					It follows from  Theorem \ref{thm76} that, for $x\in R$,
					$x=a_1+\cdots+a_N$ for some $a_i\in f(R)^{[n]}=[R, R]^{[n]}$, completing the proof.
				\end{proof}
				
				We end this section with clarifying the assumption $R=(f(R)^n)^+$ in Theorem \ref{thm76}. A natural question is whether $R=(f(R)^n)^+$ implies that
				$R=\widetilde R[R, R]\widetilde R$.
				In the proof of Theorem \ref{thm76}, Claim 4 asserts that $[R, R]\subseteq f(R)^+$.
				
				The following result connects the structure of $\widetilde{R}[R, R]\widetilde{R}$ with a simple but useful transformation of noncommutative polynomials. In the setting of noncommutative polynomials, it is often convenient to associate to a given polynomial another polynomial obtained by reorganizing the variables appearing in each monomial. More precisely, let $f$ be a polynomial in noncommuting variables. We define a new polynomial $\widetilde{f}$ by rearranging the variables in each monomial of $f$ according to a fixed standard order, typically $x_1, x_2, \ldots, x_m$, while preserving the multiplicity with which each variable appears.  For instance, if $f(x_1, x_2, x_3) = x_2^2 x_1 x_3^6 x_1 - x_3 x_1 x_2^2 x_3^5$, we obtain $\widetilde{f}(x_1, x_2, x_3) = x_1^2 x_2^2 x_3^6 - x_1 x_2^2 x_3^6.$
				
				For $a, b\in R$, we denote $a\equiv b$ if $a-b\in \widetilde{R}[R, R]\widetilde{R}$. Let $A, B$ be additive subgroups of $R$. We denote $A\equiv B$
				if and only if $A+\widetilde{R}[R, R]\widetilde{R}=B+\widetilde{R}[R, R]\widetilde{R}$. That is, every element of $A$ is equivalent to some element of $B$, and conversely.
				
				\begin{theorem}\label{thm51}
					Let $R$ be a $K$-algebra, and let $f \in F\langle \mathcal{X} \rangle$ be a polynomial. Then, for all $a_i \in R$,
					$f(a_1, a_2, \ldots, a_m) \in \widetilde{R}[R, R]\widetilde{R}$ if and only if
					$\widetilde{f}(a_1, a_2, \ldots, a_m) \in \widetilde{R}[R, R]\widetilde{R}.$
				\end{theorem}
				
				\begin{proof}
					Let $g(x_1,\ldots,x_m)x_jx_ih(x_1,\ldots,x_m)$, $i<j$, be a monomial occurring in $f$. For $a_1,\ldots,a_m\in R$, we have
					\begin{eqnarray*}
						&&g(a_1,\ldots,a_m)a_ja_ih(a_1,\ldots,a_m)\\
						&=& g(a_1,\ldots,a_m)a_ia_jh(a_1,\ldots,a_m)+g(a_1,\ldots,a_m)[a_j, a_i]h(a_1,\ldots,a_m) \\
						&\equiv&  g(a_1,\ldots,a_m)a_ia_jh(a_1,\ldots,a_m).
					\end{eqnarray*}
					By repeating the same process above,  after finitely many times we can arrive at $$f(a_1, a_2, \ldots, a_m) - \widetilde{f}(a_1, a_2, \ldots, a_m) \in \widetilde{R}[R, R]\widetilde{R},$$ that is,
					$f(a_1, a_2, \ldots, a_m) \in \widetilde{R}[R, R]\widetilde{R}$ if and only if
					$\widetilde{f}(a_1, a_2, \ldots, a_m) \in \widetilde{R}[R, R]\widetilde{R}.$
				\end{proof}

Using the connection with $\widetilde{R}[R, R]\widetilde{R}$, we continue in the spirit of Theorem~\ref{thm76}, as detailed below.

	\begin{theorem}\label{thm80}
	Let $R$ be a unital $K$-algebra and let $\displaystyle f=\sum_{\sigma\in \text{\rm Sym}(m)}\lambda_{\sigma}x_{\sigma(1)}\cdots x_{\sigma(m)}\in K\langle \mathcal{X} \rangle$ with $\displaystyle \sum_{\sigma\in \text{Sym}(m)}\lambda_{\sigma}=0$.
	If $R=(f(R)^n)^+$, where $n>1$ is a positive integer, then there exists $N\in \Bbb N$ such that, for $x\in R$,
	$x=a_1+\cdots+a_N$ for some $a_i\in f(R)^{[n]}$.
\end{theorem}

The proof of Theorem~\ref{thm80} relies on the following lemma.

				\begin{lemma}\label{lem3}
					Let $R$ be a $K$-algebra, and let $\displaystyle f=\sum_{\sigma\in \text{\rm Sym}(m)}\lambda_{\sigma}x_{\sigma(1)}\cdots x_{\sigma(m)}\in K\langle \mathcal{X} \rangle$.
					Then $f(R)^+\equiv \mu R^m$, where $\displaystyle\mu= \sum_{\sigma\in \text{Sym}(m)}\lambda_{\sigma}$. In addition, if $\mu=0$, then $f(R)^+\subseteq
					\widetilde{R}[R, R]\widetilde{R}$.
				\end{lemma}
				
				\begin{proof}
					Let $a_1,\ldots,a_m\in R$. Clearly, we have
					$$
					f(a_1,\ldots,a_m)= \sum_{\sigma\in \text{\rm Sym}(m)}\lambda_{\sigma}a_{\sigma(1)}\cdots a_{\sigma(m)}\equiv \mu a_1\cdots a_m.
					$$
					Thus $f(R)^+\equiv \mu R^m$. In addition, if $\mu=0$ then $f(R)^+\equiv  \{0\}$ and so $f(R)^+\subseteq
					\widetilde{R}[R, R]\widetilde{R}$, as claimed.
				\end{proof}
				
Equipped with Lemma~\ref{lem3}, we now turn to the proof of Theorem~\ref{thm80}.

				\begin{proof}[Proof of Theorem~\ref{thm80}]
					Taking Lemma~\ref{lem3} into account, we obtain $f(R)^+ \subseteq \widetilde{R}[R,R]\widetilde{R}$. Moreover, since $f(R) \subseteq f(R)^+$, it follows that
					$
					f(R) \subseteq \widetilde{R}[R,R]\widetilde{R},
					$
					and consequently,
					$
					(f(R)^n)^+ \subseteq \widetilde{R}[R,R]\widetilde{R}.
					$
					As $R = (f(R)^n)^+$, we deduce that
					$
					R \subseteq \widetilde{R}[R,R]\widetilde{R}.
					$
					The reverse inclusion is immediate, and hence
					$
					R = \widetilde{R}[R,R]\widetilde{R}.
					$
					In other words, $R$ is generated, as an ideal, by its commutators. Consequently, by Theorem~\ref{thm76}, there exists $N \in \mathbb{N}$ such that every $x \in R$ can be expressed in the form
					$
					x = a_1 + \cdots + a_N,
					$
					where each $a_i \in f(R)^{[n]}$, as required.
				\end{proof}
				
				Given a positive integer $m>1$, we let
				$$
				S_m(x_1,\ldots,x_m)=\sum_{\sigma\in \text{\rm Sym}(m)}(-1)^{\sigma}x_{\sigma(1)}\cdots x_{\sigma(m)}\in \Bbb Z\langle \mathcal{X} \rangle.
				$$
				denote the standard polynomial of degree $m$. Clearly, we have $\displaystyle\sum_{\sigma\in \text{\rm Sym}(m)}(-1)^{\sigma}=0$. We also note that, in a unital ring $R$, the equality $R=[R, R]^2$ if and only if $R=\widetilde R[R, R]\widetilde R$ (see \cite[Theorem 3.10]{calugareanu2026}).
				
				The following result, which generalizes Theorem \ref{thm76}, is an immediate consequence of Theorem~\ref{thm80}.
				
				\begin{theorem}\label{thm81}
					Let $R$ be a unital ring.
					If $R=(S_m(R)^n)^+$, where $m>1$ and $n>1$, then there exists $N\in \Bbb N$ such that, for $x\in R$,
					$x=a_1+\cdots+a_N$ for some $a_i\in S_m(R)^{[n]}$.
				\end{theorem}
				
				\section{Products of additive commutators}\label{sec product}

				In light of Theorem~\ref{thm77} and recent studies, we are concerned with the question of whether every matrix over a given ring can be written as a product of additive commutators. This problem has a long history, beginning with the case where \( R \) is a field. In that setting, it was shown by Botha in \cite[Theorem 4.1]{Pa_Botha_97} that every matrix in \( \mathrm{M}_n(F) \) is a product of two  commutators; for fields of characteristic zero, the result had appeared even earlier in \cite{Pa_Wu_89} by Wu. More recently, the problem has been extended to noncommutative settings. When \( R = D \) is a division ring, it was shown in \cite[Theorem ~4.4]{Pa_SoDu_25} that every element of \( \mathrm{M}_n(D) \) can be expressed as a product of at most seven commutators. This bound was subsequently improved to two under the additional assumption that the center of \( D \) is infinite \cite[Theorem ~4.4]{Pa_Bre_25}. Interestingly, when \( n = 2 \) or when the matrix is singular, the assumption on the center can be dropped altogether \cite[Propositions~4.2 and~4.7]{Pa_Bre_25}. Moreover, the question has been examined in more general rings beyond division rings. In particular, for rings of Bass stable rank one, every matrix is a product of three commutators \cite[Theorem 3.9]{Pa_Bre_25}. However, the phenomenon does not hold universally: there exist commutative unital rings \( R \) and matrices in \(  \mathrm{M}_2(R) \) that cannot be written as a product of commutators \cite[Theorem~2.2]{Pa_Bre_25}.
				
				Taken together, these results highlight an intriguing open question: does every nonsingular matrix over a division ring with finite center admit a factorization as a product of two  commutators? This section is devoted to further exploring this problem and related directions.
				
				We assume familiarity with basic ring theory. Throughout this section, $D$ denotes a division ring with center $F$, and $\mathrm{M}_n(D)$ denotes the ring of $n\times n$ matrices over $D$, where $n\geq1$ is an integer. More generally, for any positive integers $t$ and $m$, we use $\mathrm{M}_{t \times m}(D)$ to refer to the set of all $t \times m$ matrices over $D$.  In addition, we write $\mathrm{LT}_n(D)$ and $\mathrm{UT}_n(D)$ for the sets of lower and upper unitriangular matrices in $\mathrm{M}_n(D)$, respectively; these are the lower (respectively, upper) triangular matrices with all diagonal entries equal to~$1$. The notation $\mathrm{I}_n$ is used for the identity matrix in the algebra $\mathrm{M}_n(D)$.

				Let \( A = (a_{ij}) \in \mathrm{M}_k(D) \) and \( B = (b_{ij}) \in \mathrm{M}_h(D) \), where \( k \) and \( h \) are positive integers. We define their \emph{block diagonal sum} as
				\[
				A \oplus B := \begin{pmatrix}
					A & 0 \\
					0 & B
				\end{pmatrix} \in \mathrm{M}_{k+h}(D),
				\]
				that is, the block diagonal matrix with \(A\) and \(B\) along the diagonal and zeros elsewhere. More generally, for a finite family of matrices \( \{ A_i \in \mathrm{M}_{n_i}(D) \mid 1 \leq i \leq s \} \), where each \(n_i\) is a positive integer, we write
				\[
				\bigoplus_{i=1}^s A_i := A_1 \oplus A_2 \oplus \cdots \oplus A_s
				\]
				for their successive block diagonal sum. In particular, if each \( A_i \in D \cong \mathrm{M}_1(D) \), we use the notation
				\[
				\mathrm{diag}(A_1, \dots, A_s) := \bigoplus_{i=1}^s A_i.
				\]

				For the reader’s convenience, we reiterate the recent results on $\mathrm{M}_n(D)$ that have been mentioned above, based on \cite[Theorem 4.1]{Pa_Botha_97}, and \cite[Lemma 4.3, Theorems 3.9 and 4.4, Propositions 4.2 and 4.7]{Pa_Bre_25}.
				
				\begin{theorem}\label{recent}
					Let $D$ be a division ring with center $F$, and let $n \geq 2$ be an integer. Then every matrix $A \in \mathrm{M}_n(D)$ can be written as a product of at most three commutators in $\mathrm{M}_n(D)$, and of two commutators in $\mathrm{M}_n(D)$ if one of the following holds:
					\begin{enumerate}[\rm (i)]
						\item $D = F$;
						\item $n = 2$;
						\item  $A=\lambda\mathrm{I}_n$ for some $\lambda\in F$ and $D$ contains at least three elements;
						\item $A$ is singular;
						\item $D \ne F$ and $F$ is infinite;
						\item $D \ne F$ and $\dim_F D < \infty$.
					\end{enumerate} Furthermore, in cases (ii), (iii), (v), and (vi), the first factor turns out to be nonsingular.
				\end{theorem}

				Theorem~\ref{recent} reduces the discussion to a single outstanding case: $n\geq3$, $A$ is nonsingular, $D \ne F$, the center $F$ is finite, and $D$ is infinite-dimensional over $F$. This exceptional situation will be the primary focus of what follows. To proceed, we formulate the following guiding question:
				
				\begin{question}\label{que}
					Let $n\geq3$ be an integer, and let \( D \) be a division ring with finite center \( F \), such that \( D \) is infinite-dimensional over \( F \), and let \( A \in \mathrm{M}_n(D) \) be a nonsingular matrix. Can \( A \) be written as a product of two  commutators in \( \mathrm{M}_n(D) \)?
				\end{question}
				
				Moreover, division rings satisfying the conditions in Question~\ref{que} do exist. For concrete examples and further details, we refer the reader to \cite[$\S4$, Example~3]{Pa_ChuLee_79} and \cite[Proposition~2.3.5]{Bo_Cohn_95}.
				
				\subsection{On the finiteness of the center of a division ring}
				
				We now turn to the case where the center of the division ring is finite, in line with the focus stated in Question~\ref{que}. In particular, we establish the following result.
				
				\begin{theorem} \label{q>n}
					Let $D$ be a division ring with center $F$, and let $n \ge 2$ be an integer. Suppose $F$ is a finite field with $q$ elements, where $q > n$. Then every matrix in $\mathrm{M}_n(D)$ can be expressed as a product of two commutators in $\mathrm{M}_n(D)$.
				\end{theorem}
				
				To prove Theorem~\ref{q>n}, we adapt the main idea from the proof of \cite[Theorem~4.4]{Pa_Bre_25}, with certain modifications, and begin by establishing a sequence of lemmas. The first lemma relies on the fact, proved in \cite[Proposition 1.8]{Pa_AmiRow_94}, that every matrix in $\mathrm{M}_n(D)$ is similar to one whose $(1,1)$-entry is zero.  It is worth noting that, although \cite{Pa_AmiRow_94} works under the standing assumption that division rings are finite-dimensional over their centers, this hypothesis is not actually needed in the proof of \cite[Proposition 1.8]{Pa_AmiRow_94}.
				
				\begin{lemma} {\rm \cite[Proposition 1.8]{Pa_AmiRow_94}}\label{Rowen}
					Let $D$ be a division ring and let $n\geq2$ be an integer. Then, every noncentral matrix in $\mathrm{M}_n(D)$ is similar to a matrix of the form: $\begin{pmatrix}
						0&B\\
						C&E
					\end{pmatrix},$ where $B\in\mathrm{M}_{1\times (n-1)}(D), C\in\mathrm{M}_{(n-1)\times1}(D)$ and $E\in\mathrm{M}_{n-1}(D)$.
				\end{lemma}
				
				Following once more the approach in \cite[Theorem 4.4]{Pa_Bre_25}, the final step is to show that for every $A \in \mathrm{M}_n(D)$ there exists a scalar $\lambda \in F$ such that $A - \lambda \mathrm{I}_n$ is nonsingular. This leads us to the following lemma, stated in the setting where the center of $D$ is finite, as required in Question~\ref{que}.
				
				\begin{lemma}\label{lemma in}
					Let $D$ be a division ring with center $F$, and let $n \ge 2$ be an integer. Suppose $F$ is a finite field with $q$ elements. If $q > n$, then for every $A \in \mathrm{M}_n(D)$, there exists a scalar $\lambda \in F$ such that $A - \lambda \mathrm{I}_n$ is nonsingular.
				\end{lemma}

				\begin{proof}
					Assume, toward a contradiction, that $A - \lambda \mathrm{I}_n$ is singular for every $\lambda \in F$. Then, for each $\lambda \in F$ we can choose a nonzero column vector $v_\lambda \in \mathrm{M}_{n\times 1}(D)$ such that $$(A - \lambda \mathrm{I}_n) v_\lambda = 0,$$ which is equivalent to $
					A v_\lambda = \lambda v_\lambda.
					$ Since $\lambda \in F$ lies in the center of $D$, it commutes with every element of $D$, and thus $
					\lambda v_\lambda = v_\lambda \lambda.
					$
					
					Now pick distinct scalars $\lambda_1, \dots, \lambda_m \in F$ with $m \ge 2$, and consider the associated nonzero vectors $
					v_{\lambda_1}, \dots, v_{\lambda_m} \in \mathrm{M}_{n\times 1}(D).
					$ 	We claim these vectors are right $D$-linearly independent.
					
					Suppose
					\begin{equation}\label{lin-rel}
						v_{\lambda_1} r_1 + \cdots + v_{\lambda_m} r_m = 0
					\end{equation}
					for some $r_1, \dots, r_m \in D$. Applying $A$ to Eq.\eqref{lin-rel} and using $A v_{\lambda_i} = -\lambda_i v_{\lambda_i}$ gives
					\begin{equation}
						0 = A\left(\sum_{i=1}^m v_{\lambda_i} r_i \right)
						= \sum_{i=1}^m (\lambda_i v_{\lambda_i}) r_i
						= \sum_{i=1}^m v_{\lambda_i} (\lambda_i r_i). \label{new}
					\end{equation}
					Subtracting $\lambda_1$ times Eq.\eqref{lin-rel} from Eq.\eqref{new} yields
					\[
					\sum_{i=1}^m v_{\lambda_i} (\lambda_i - \lambda_1) r_i = 0.
					\]
					For $i > 1$, $\lambda_i - \lambda_1 \in F$ is a nonzero central element of $D$, hence invertible. Setting $s_i = (\lambda_i - \lambda_1) r_i$ for $i>1$, we obtain
					\[
					\sum_{i=2}^m v_{\lambda_i} s_i = 0.
					\]
					The same elimination argument can now be repeated: successively remove one vector at a time until only a single term remains, which must be zero. Tracing back, we find all $r_i = 0$. Thus  $v_{\lambda_1}, \dots, v_{\lambda_m}$
					are right $D$-linearly independent.
					
					In particular, if we take $m = q$, we produce $q$ right $D$-linearly independent vectors in $\mathrm{M}_{n\times 1}(D)$. But this $D$-module has dimension $n$, so it cannot contain more than $n$ independent vectors. Therefore $q \le n$, contradicting our assumption $q > n$.	It follows that there must exist $\lambda \in F$ for which $A - \lambda \mathrm{I}_n$ is invertible.
				\end{proof}
				
				\begin{remark}
					Lemma~\ref{lemma in} fails if $q \le n$. Here is a counterexample. Let $F$ be the finite field with $q$ elements, and consider the polynomial
					\[
					g(x) = x^q - x \in F[x].
					\]
					Since $F\setminus\{0\}$ is a cyclic group under multiplication of $F$, it follows that $g(\lambda) = 0$ for every $\lambda \in F$. Let $C \in \mathrm{M}_q(F)$ be the companion matrix of $g(x)$. By construction,
					\[
					\det(\lambda \mathrm{I}_q - C) = g(\lambda) = 0
					\]
					for all $\lambda \in F$, so $C - \lambda \mathrm{I}_q$ is singular for every $\lambda$.
					
					If $n = q$, take $A = C$ to obtain a counterexample. If $n > q$, extend $C$ by a zero block:
					\[
					A = \begin{pmatrix}
						C & 0 \\[3pt]
						0 & 0_{n-q}
					\end{pmatrix} \in \mathrm{M}_n(F) \subseteq \mathrm{M}_n(D).
					\]
					Then $A - \lambda \mathrm{I}_n$ remains singular for all $\lambda \in F$.
				\end{remark}
				
				\begin{remark}
					A key point in the proof of Lemma~\ref{lemma in} is that we are working in the column space
					\(\mathrm{M}_{n\times 1}(D)\) with its natural structure as a \emph{right \(D\)-module}.
					The argument relies on \emph{right \(D\)-linear independence}, not on \(F\)-linear independence.
					Indeed, if one were to work over \(F\) instead, the dimension would be
					\[
					\dim_F \mathrm{M}_{n\times 1}(D) =
					\begin{cases}
						n \cdot \dim_F D& \text{if } \dim_F D < \infty, \\
						\infty& \text{if } \dim_F D = \infty,
					\end{cases}
					\]
					so in the infinite-dimensional case there would be no restriction on the number of
					\(F\)-linearly independent columns.
					In contrast, as a right \(D\)-module, \(\mathrm{M}_{n\times 1}(D)\) has dimension exactly \(n\),
					which is what makes the counting argument work.
				\end{remark}
				
				We are now in a position to present the proof of Theorem~\ref{q>n}.

				\begin{proof}[Proof of Theorem~\ref{q>n}]
					We are guided by the idea in the proof of \cite[Theorem 4.4]{Pa_Bre_25}, making some adjustments for our setting.  By Theorem~\ref{recent}, we may restrict to the case where $D$ is noncommutative and $F$ is finite.  We argue by induction on $n$, proving the stronger claim:  Every matrix in $\mathrm{M}_n(D)$ is expressible as a product of two commutators, with the first factor being invertible.
					
					Let $A \in \mathrm{M}_n(D)$.  We begin with the central case. First, if $A$ lies in the center of $\mathrm{M}_n(D)$, then $A \in \mathrm{M}_n(F)$. By Theorem~\ref{recent}, $A$ is already a product of two commutators in $\mathrm{M}_n(F)$, with the first one invertible, and hence also in $\mathrm{M}_n(D)$. Now, we use reduction to the noncentral case. Assume now that $A$ is noncentral.	The base case $n = 2$ is covered by Theorem~\ref{recent}. We continue to the  inductive step.
					
					Let $n \ge 3$ and suppose the claim holds for all smaller sizes. By Lemma~\ref{Rowen}, $A$ is similar to a block matrix
					$
					\begin{pmatrix}
						0 & B \\
						C & E
					\end{pmatrix},
					$
					where $B \in \mathrm{M}_{1 \times (n-1)}(D)$, $C \in \mathrm{M}_{(n-1) \times 1}(D)$, and $E \in \mathrm{M}_{n-1}(D)$. As the desired conclusion is preserved under similarity, we may, without loss of generality, assume that $A$ already has this block form.
					
					We next set up the first commutator.  Because $D$ is noncommutative, there exists a nonzero commutator $d = [d_1, d_2] \in D$ where $d_1,d_2\in D$. By the inductive hypothesis, we may write
					$
					E = [E_1, E_2] \cdot [E_3, E_4],
					$
					where $E_1,E_2,E_3,E_4\in\mathrm{M}_{n-1}(D)$ and  $[E_1, E_2]$ is nonsingular. Then
					\[
					A =
					\begin{pmatrix}
						0 & B \\
						C & [E_1, E_2][E_3, E_4]
					\end{pmatrix}
					=
					\underbrace{\begin{pmatrix}
							[d_1, d_2] & 0 \\
							0 & [E_1, E_2]
					\end{pmatrix}}_{\text{first factor}}
					\cdot
					\underbrace{\begin{pmatrix}
							0 & d^{-1}B \\
							[E_1, E_2]^{-1}C & [E_3, E_4]
					\end{pmatrix}}_{\text{second factor}}.
					\]
					The first factor is indeed a commutator:
					\[
					\begin{pmatrix}
						[d_1, d_2] & 0 \\
						0 & [E_1, E_2]
					\end{pmatrix}
					=
					\left[
					\begin{pmatrix}
						d_1 & 0 \\
						0 & E_1
					\end{pmatrix},
					\begin{pmatrix}
						d_2 & 0 \\
						0 & E_2
					\end{pmatrix}
					\right].
					\] Note that the first factor is nonsingular, as promised.
					
					We now write the second factor as a commutator.
					Since $q > n$, Lemma~\ref{lemma in} guarantees the existence of $\lambda \in F$ such that $E_3 - \lambda \mathrm{I}_{n-1}$ is nonsingular. For this choice of $\lambda$, it is not difficult to verify that
					\begin{eqnarray*}
						&&	\begin{pmatrix}
							0 & d^{-1}B \\
							[E_1, E_2]^{-1}C & [E_3, E_4]
						\end{pmatrix}\\
						&=&
						\left[
						\begin{pmatrix}
							\lambda & 0 \\
							0 & E_3
						\end{pmatrix},
						\begin{pmatrix}
							0 & -d^{-1}B(E_3 - \lambda \mathrm{I}_{n-1})^{-1} \\
							(E_3 - \lambda \mathrm{I}_{n-1})^{-1}[E_1, E_2]^{-1}C & E_4
						\end{pmatrix}
						\right].
					\end{eqnarray*}
					
					Finally, we have expressed $A$ as a product of two commutators, with the first commutator invertible. This completes the inductive step, and hence the proof.
				\end{proof}
				
				\subsection{Products of commutators in division rings}
				
				We gather information about $D$ by examining the question of whether every element of a division ring \( D \) can be written as a product of two commutators in \( D \). The first known example of a division ring in which every element is itself a commutator was constructed by Harris in \cite{Pa_Ha_58}. This line of inquiry was further developed by Makar-Limanov in \cite{Pa_Ma_89}. More recently, Lichtman showed in \cite{Pa_Li_05} that the phenomenon observed by Harris is not rare. In particular, he proved that any division ring \( D \) can be embedded into a larger division ring \( K \) in which every element is a commutator \cite[Corollary]{Pa_Li_05}. This implies that given any division ring \( D \), there exists an extension division ring \( K \supseteq D \) such that every element of \( K \) can be written as a product of two commutators within \( K \).  Besides, it turns out that this question is also raised explicitly in \cite[Question~5.9]{Pa_Ga_25}. Recently in \cite{Pa_JaKe_25}, Jang and Ke gave an affirmative answer in the case where \( D \) is a skew Laurent series division ring over a field. In particular, it is shown in \cite[Theorem]{Pa_JaKe_25} that if $D$ is a skew Laurent series division ring over a field, then every element of $D$ can be expressed as a product of two commutators.

				Relating to skew Laurent series division rings, we proceed with the following observation. As shown in \cite[Proposition 2.3.5]{Bo_Cohn_95}, for any given field \( F \), one can construct a noncommutative division ring \( D \) whose center is precisely \( F \). It is well known that any division ring is a vector space over its center, and in Cohn's construction, the resulting division ring has infinite dimension over \( F \). A closer inspection of the proof reveals that \( D \) arises as a skew Laurent series division ring. Combining this insight with the recent result of Jang and Ke in \cite[Theorem]{Pa_JaKe_25}, which shows that every element in such a division ring is a product of two commutators, we obtain the following:
				
				\begin{theorem}\label{exist}
					Let \( F \) be a field. Then there exists a noncommutative division ring \( D \) with center \( F \), such that \( D \) is infinite-dimensional over \( F \), and every (respectively, nonzero) element of \( D \) can be written as a product of two (respectively, noncentral) commutators in~\( D \).
				\end{theorem}
				
				Here, noncentral commutators refer to commutators that do not belong to the center of the division ring. Although \cite[Theorem]{Pa_JaKe_25} only states that every element is a product of two commutators, a closer inspection of the construction in \cite[Section 2.3]{Pa_JaKe_25} reveals that these commutators are in fact noncentral in the case of a skew Laurent series division ring which is infinite-dimensional over its center. We are grateful to Hau-Yuan Jang  (coauthor of \cite{Pa_JaKe_25}) for confirming this observation.

				In Theorem~\ref{exist}, the division ring \( D \) constructed is noncommutative and infinite-dimensional over its center, which is exactly \( F \). However, there are also noncommutative division rings that are finite-dimensional over some field, although their center may not coincide with the given field \( F \). Some of these examples still arise from skew Laurent series division rings, provided certain conditions are imposed on the underlying field automorphism. See \cite{Pa_JaKe_25} in detail.
				
				On the other hand, not all such division rings need to come from skew Laurent series constructions. For example, it has been shown in \cite[Theorem 5.1]{Pa_DaDuSo_24} and \cite[Theorem 3.1]{Pa_DaDuSo_25} that in any quaternion division algebra over a Pythagorean field, every nonzero element can be written as a product of two noncentral commutators.

				As noted in Theorem~\ref{recent}, the remaining case to consider is when \( A \) is nonsingular, the center \( F \) is a finite field, and the division ring \( D \) has infinite dimension over \( F \). According to Theorem~\ref{exist}, such division rings do exist. In particular, there are noncommutative division rings \( D \) with finite center $F$ that are infinite-dimensional over \( F \) and in which every nonzero element can be written as a product of two noncentral commutators. With this in mind, we now proceed under the additional assumption that \( D \) admits such a decomposition for all its elements, thereby allowing us to give an affirmative answer to Question~\ref{que} in this setting.

				\begin{theorem}\label{mainthm}
					Let $D$ be a division ring with center $F$, and let $n\geq 2$ be an integer such that either
					\begin{enumerate}[\rm (i)]
						\item  the field $F$ has at least three elements, or
						\item the field $F$ has exactly two elements  and $n$ is odd.
					\end{enumerate}		If  every nonzero element of $D$ can be expressed as a product of  two noncentral commutators in $D$, then any matrix in $\mathrm{M}_n(D)$ can be written as a product of two  commutators.
				\end{theorem}
	
				To proceed with the proof of Theorem~\ref{mainthm}, we first present a sequence of auxiliary lemmas. When dealing with nonsingular matrices, the following result will be a key tool.
				
				\begin{lemma} {\rm \cite[Lemma~2.1]{Egorchenkova}} \label{LHU}
					Let \( D \) be a division ring with center $F$ and let \( n\geq2 \) be an integer. Suppose that $h_1,h_2,\ldots,h_{n-1}\in D\setminus\{0\}$. If $A\in\mathrm{M}_n(D)$ is nonsingular and not in \( \{ \lambda \mathrm{I}_n \mid \lambda \in F \} \) where $\mathrm{I}_n$ is the identity matrix in $\mathrm{M}_n(D)$, then there exists an nonsingular matrix \( P \in \mathrm{M}_n(D) \) such that
					$P^{-1} A P = L H U,$
					where \( L \in \mathrm{LT}_n(D) \), \( U \in \mathrm{UT}_n(D) \), and \( H = \mathrm{diag}(h_1,h_2,\ldots,h_{n-1},h_n) \) for some nonzero element \( h_n \in D\).
				\end{lemma}

				Motivated by the structure of triangular matrices in Lemma~\ref{LHU}, we now examine how such matrices can be expressed as commutators. A matrix $(a_{jk}) \in \mathrm{M}_n(R)$ is said to have \emph{zero trace} if the sum of its diagonal entries, $a_{11} + \cdots + a_{nn}$, equals zero. The next result, which is known (see, e.g., \cite[Theorem~4]{Pa_Ka_14}), confirms that every triangular matrix with zero trace arises as a commutator.
				
				\begin{theorem}
					\label{triangular trace zero}
					Let $R$ be a unital associative ring and let $n \geq 2$ be an integer. Then every triangular matrix in $\mathrm{M}_n(R)$ with zero trace is a commutator.
				\end{theorem}

				To apply Theorem~\ref{triangular trace zero} effectively, we next establish that a trace-zero diagonal matrix with nonzero entries can indeed be constructed under mild assumptions on the field.
				
				\begin{lemma}\label{choose}
					Let $F$ be a field and let $n > 1$ be an integer. Then there exist nonzero elements  $x_1, x_2, \dots, x_n \in F$ such that $x_1 + x_2 + \cdots + x_n = 0$ in either of the following cases:
					\begin{enumerate}[\rm (i)]
						\item The field $F$ has at least three elements;
						\item The field $F$ has exactly two elements  and $n$ is even.
					\end{enumerate}
				\end{lemma}
				
				\begin{proof}
					If $F$ has only two elements, say $F = \{0,1\}$, then the only nonzero element is $1$, and any $x_i \ne 0$ must be equal to $1$. In this case, the sum $x_1 + \dots + x_n$ is equal to $n \cdot 1$, which is $0$ if and only if $n$ is even. Hence, the result holds in this case precisely when $n$ is even.
					
					Now assume $F$ has at least three elements. Then  $F\setminus\{0\}$ contains at least two distinct elements. Choose $1$ and another nonzero element $a \ne 1$. Set
					$x_1 = \dots = x_{n-2} = 1$ and $x_{n-1} = a.$
					Then the partial sum is
					$S = (n-2)\cdot1 + a = n - 2 + a.$
					We avoid the unique $a$ that would make $S = 0$ by choosing $a \ne -(n-2)$. This is possible since $F$ has more than two elements. Then define
					$x_n = -\bigl((n - 2) + a\bigr).$
					By construction, $x_n \ne 0$ and the full sum satisfies
					$x_1 + \cdots + x_n = 0.$
					This completes the proof.
				\end{proof}

				The subsequent key step is to determine whether the matrix $\begin{pmatrix}
					B & \alpha \\
					0 & a
				\end{pmatrix}$
				is similar to
				$\begin{pmatrix}
					B & 0 \\
					0 & a
				\end{pmatrix},$	where \( B \in \mathrm{M}_n(D) \), \( \alpha \in \mathrm{M}_{n \times 1}(D) \), and \( a \in D \). To address this question, suppose there exists \( y \in \mathrm{M}_{n \times 1}(D) \) such that
				\[
				\begin{pmatrix}
					\mathrm{I}_n & y \\
					0 & 1
				\end{pmatrix}
				\begin{pmatrix}
					B & \alpha \\
					0 & a
				\end{pmatrix}
				\begin{pmatrix}
					\mathrm{I}_n & y \\
					0 & 1
				\end{pmatrix}^{-1}
				=
				\begin{pmatrix}
					B & 0 \\
					0 & a
				\end{pmatrix}.
				\]
				A straightforward computation shows that this is equivalent to the condition $ \alpha = By - ya.$ Hence, the problem reduces to the existence of a solution \( y \in \mathrm{M}_{n \times 1}(D) \) to the equation $By - ya = \alpha.$
				This is a special case of the Sylvester equation over a division ring. A sufficient condition for the solvability of this equation is provided by the following result of Bolotnikov.
				
				\begin{lemma} {\rm \cite[Lemma 3.1]{Bo_Bo_22}} \label{lem Bo}
					Let $D$ be a division ring with center $F$, and let \( A \in \mathrm{M}_n(D) \) and \( B \in \mathrm{M}_m(D) \), where \( n, m \) are positive integers. Suppose there exists a polynomial \( p \) in one variable with coefficients in $F$ such that \( p(A) = 0 \) and \( p(B) \) is invertible. Then, for any \( C \in \mathrm{M}_{n \times m}(D) \), the Sylvester equation
					$AX - XB = C$
					has a unique solution \( X \in \mathrm{M}_{n \times m}(D) \).
				\end{lemma}
				
				\begin{remark}
					Lemma~\ref{lem Bo} is a particular case of a more general result due to Cohn in \cite[Lemma 2.3]{Pa_Co_73}, which addresses algebras over a common field. However, since we are working specifically with matrices, we state the result in this more accessible matrix form.
				\end{remark}
				
				By applying Lemma~\ref{lem Bo}, we obtain the following criterion for similarity:
				
				\begin{lemma}\label{cheo khoi}
					Let $D$ be a division ring with center $F$, and let \( B \in \mathrm{M}_n(D) \), \( \alpha \in \mathrm{M}_{n \times 1}(D) \), and \( a \in D \). If there exists a polynomial \( p \) with coefficients in $F$ such that \( p(B) = 0 \) and \( p(a) \) is invertible, then $\begin{pmatrix}
						B & \alpha \\
						0 & a
					\end{pmatrix}$ is similar to
					$\begin{pmatrix}
						B & 0 \\
						0 & a
					\end{pmatrix}.$
				\end{lemma}
				
				With the foundational steps completed, we now turn to the proof of Theorem~\ref{mainthm}
				
				\begin{proof}[Proof of Theorem~\ref{mainthm}]
					Let $A \in \mathrm{M}_n(D)$. By Theorem~\ref{recent}, it suffices to consider the case when $A$ is nonsingular. If $A = \lambda \mathrm{I}_n$ for some $\lambda \in F\setminus\{0\}$, then since every element of $D$ is a product of two noncentral commutators in $D$, it follows that $\lambda = [a,b]\cdot [c,d]$ for some $a,b,c,d \in D$, and hence $A = [a \mathrm{I}_n, b \mathrm{I}_n]\cdot[c \mathrm{I}_n, d \mathrm{I}_n]$, as desired. Otherwise, assume $A \notin \{ \lambda \mathrm{I}_n \mid \lambda \in F \}$.
					
					Again by Theorem~\ref{recent}, it remains to address the case where $D$ is infinite-dimensional over its center $F$, and $n \geq 3$. Taking Lemma~\ref{choose} into account, we may choose nonzero elements $x_1, x_2, \dots, x_{n-1} \in F$ such that $x_1 + x_2 + \cdots + x_{n-1} = 0$. Define $h_i = x_i^2$ for $i \in \{1, \dots, n-1\}$. By Lemma~\ref{LHU}, there exists a nonsingular matrix $P \in \mathrm{M}_n(D)$ such that
					$
					P^{-1} A P = LHU,
					$ where $L \in \mathrm{LT}_n(D)$, $U \in \mathrm{UT}_n(D)$, and $H = \mathrm{diag}(h_1, \dots, h_{n-1}, h_n)$ for some nonzero $h_n \in D$.
					
					Since every element of $D$ is a product of two noncentral commutators in $D$, we can write $h_n = h_n' h_n''$, where $h_n', h_n'' \in D \setminus F$ are noncentral commutators. In what follows, we define the diagonal matrices $H' = \mathrm{diag}(x_1, \dots, x_{n-1}, h_n')$ and $H'' = \mathrm{diag}(x_1, \dots, x_{n-1}, h_n'')$, and set $L_1 = L H'$, $U_1 = H'' U$. Then, it is not difficult to verify  $
					P^{-1} A P = L_1 U_1.
					$ We can decompose $L_1$ and $U_1$ as:
					$$
					L_1 = \begin{pmatrix}
						L_2 & 0 \\
						\ell & h_n'
					\end{pmatrix}, \quad
					U_1 = \begin{pmatrix}
						U_2 & \mathbf{u} \\
						0 & h_n''
					\end{pmatrix},
					$$ where $L_2, U_2 \in \mathrm{M}_{n-1}(D)$ are lower and upper triangular matrices, respectively, with $x_1, \dots, x_{n-1}$ on their diagonals, and  $\mathbf{u}\in\mathrm{M}_{(n-1)\times 1}(D),\ell\in\mathrm{M}_{1\times (n-1)}(D)$. Moreover, since $x_1 + \cdots + x_{n-1} = 0$, it follows from Theorem~\ref{triangular trace zero} that $L_2$ and $U_2$ are additive commutators: $L_2 = [L_3, L_4]$ and $U_2 = [U_3, U_4]$ for some $L_3,L_4,U_3,U_4\in\mathrm{M}_{n-1}(D)$.
					
					Now, consider the polynomial $p(x) = (x - x_1)(x - x_2) \cdots (x - x_{n-1})$ in the variable $x$ with coefficients in $F$. This polynomial vanishes on both $L_2$ and $U_2$, but evaluates to an invertible element at both $h_n'$ and $h_n''$, since these are noncentral. By Lemma~\ref{cheo khoi}, $L_1$ is similar to $L_2 \oplus (h_n')$, and $U_1$ is similar to $U_2 \oplus (h_n'')$. Thus, there exist nonsingular matrices $P_1, P_2 \in \mathrm{M}_n(D)$ such that:
					$$
					P_1^{-1} L_1 P_1 = L_2 \oplus (h_n'), \quad P_2^{-1} U_1 P_2 = U_2 \oplus (h_n'').
					$$ Since $h_n' = [h_{n_1}', h_{n_2}']$ and $h_n'' = [h_{n_1}'', h_{n_2}'']$ for some $h_{n_1}',h_{n_2}',h_{n_1}'',h_{n_2}''\in D$, we can write:
					$$
					L_2 \oplus (h_n') = [L_3 \oplus (h_{n_1}'), L_4 \oplus (h_{n_2}')], \quad
					U_2 \oplus (h_n'') = [U_3 \oplus (h_{n_1}''), U_4 \oplus (h_{n_2}'')].
					$$ Conjugating back via $P_1$ and $P_2$, we obtain: \begin{eqnarray*}
						L_1 &=& [P_1 (L_3 \oplus (h_{n_1}')) P_1^{-1}, P_1 (L_4 \oplus (h_{n_2}')) P_1^{-1}],\\
						U_1 &=& [P_2 (U_3 \oplus (h_{n_1}'')) P_2^{-1}, P_2 (U_4 \oplus (h_{n_2}'')) P_2^{-1}].	
					\end{eqnarray*} Therefore, putting everything together, we conclude that
					\begin{eqnarray*}
						A&=&PL_1U_1P^{-1}=PL_1P^{-1}\cdot PU_1P^{-1}\\
						&=&[PP_1(L_3\oplus(h_{n_1}'))(PP_1)^{-1},PP_1(L_4\oplus (h_{n_2}'))(PP_1)^{-1}]\\
						&& [PP_2(U_3\oplus(h_{n_1}''))(PP_2)^{-1},PP_2(U_4\oplus (h_{n_2}''))(PP_2)^{-1}]
					\end{eqnarray*} is a product of two commutators in $\mathrm{M}_n(D)$, as claimed.
				\end{proof}
				
				\subsection{Algebraicity of division rings}
				
			  In Theorem~\ref{recent}, the finite-dimensional case was already addressed in \cite[Corollary~4.5]{Pa_Bre_25}, which is a consequence of  \cite[Theorem~4.4]{Pa_Bre_25}.
				Here, we again draw on \cite[Theorem~4.4]{Pa_Bre_25} to extend the conclusion of  \cite[Corollary~4.5]{Pa_Bre_25} from finite-dimensional division rings to a broader  class, namely, algebraic division rings.  Recall that a division ring $D$ is called \emph{algebraic} if every element $a \in D$ is algebraic over  its center $F$; that is, there exists a nonzero polynomial $f$ in one variable with coefficients in $F$ such that $f(a) = 0$.

				\begin{theorem} \label{algebraic}
					If $D$ is an algebraic division ring and $n\geq 2$ is an integer, then every matrix in $\mathrm{M}_n(D)$ can be expressed as a product of two commutators in $\mathrm{M}_n(D)$.
				\end{theorem}
				
				\begin{proof}
					If $D$ is a field, the result follows from Theorem~\ref{recent}. We now consider the case where $D$ is a noncommutative division ring. Let $F$ denote the center of $D$. If $F$ is infinite, then the conclusion follows from Theorem~\ref{recent}. On the other hand, if $F$ is finite, then by Jacobson's theorem \cite[(13.11) Theorem]{Bo_Lam_91}, $D$ must be commutative, which is a contradiction. This completes the proof.
				\end{proof}
				
				In the work \cite{Pa_Mad_00}, Mahdavi-Hezavehi pointed out that Jacobson \cite{Bo_Ja_56} had already raised the following question: if $D$ is an algebraic division ring with center $F$, must every matrix in $\mathrm{M}_n(D)$ be algebraic over $F$? To address this problem, Mahdavi-Hezavehi outlined a possible strategy (see \cite[Corollaries~5.16 and~5.18]{Pa_Mad_00}).
				
				Independently of this line of thought, it is known from \cite[Theorem~3.2]{Pa_Co_73} that every matrix in $\mathrm{M}_n(D)$ can be similar to a block diagonal form consisting of an algebraic part and a transcendental part over $F$. Recall that a matrix $A \in \mathrm{M}_n(D)$ is called \emph{algebraic} over $F$ if it satisfies a polynomial equation in one variable with coefficients in $F$. On the other hand, if $f(A)$ is nonsingular for every nonzero polynomial $f$ in one variable with coefficients in $F$, then $A$ is said to be \emph{transcendental} over $F$.
				
				For the transcendental case, it is shown in \cite[Corollary, Section~3]{Pa_Co_73} that such a matrix is necessarily cyclic, and hence similar to a companion matrix. Moreover, \cite[Corollary~1, Section~4]{Pa_Co_73} further implies that there is a suitable extension division ring $K$ of $D$ such that every transcendental matrix is similar to a scalar matrix $\lambda \mathrm{I}_n$ with $\lambda \in K$.
				
			Turning to the algebraic setting, results of Djokovi\'c \cite[Proposition~6 and Theorem~7]{Pa_Dra_85} show that if a division ring \( D \) contains an algebraic closure of its center \( F \), then every matrix in \( \mathrm{M}_n(D) \) that is algebraic over \( F \) admits a generalized Jordan canonical form in \( \mathrm{M}_n(D) \). In particular, under this assumption, every matrix in \( \mathrm{M}_n(D) \) possesses a generalized Jordan normal form. Here, a generalized Jordan normal form in \( \mathrm{M}_n(D) \) refers to a matrix of the form
			\[
			J_{n}(\alpha,\beta) =
			\begin{pmatrix}
				\alpha & \beta & 0 & \cdots & 0 \\
				0 & \alpha & \beta & \cdots & 0 \\
				0 & 0 & \alpha & \ddots & \vdots \\
				\vdots & \vdots & \ddots & \ddots & \beta \\
				0 & 0 & \cdots & 0 & \alpha
			\end{pmatrix}
			\in \mathrm{M}_n(D).
			\]

			By the way, division rings \( D \) that contain an algebraic closure of their center \( F \) are well understood in the finite-dimensional case. Indeed, it is shown in \cite[(15.9) Theorem, Page 243]{Bo_Lam_91} that  if \( \dim_F D < \infty \), then \( D \) is necessarily a quaternion division algebra over \( F \), and \( F \) must be a real closed field. The infinite-dimensional case, where \( \dim_F D = \infty \), remains open.

			Moreover, it is shown in \cite[Theorem~7]{Pa_Dra_85} that if \( \alpha \) is separable over the center \( F \), then one may take \( \beta = 1 \). This observation suggests that, instead of assuming that \( D \) contains an algebraic closure of \( F \), it is natural to consider the case where \( F \) itself is algebraically closed. In this situation, for every matrix \( A \in \mathrm{M}_n(D) \), there exists \( P \in \mathrm{GL}_n(D) \) such that
			\[
			P^{-1} A P = \bigoplus_{i=1}^s J_{m_i}(\alpha_i,1),
			\]
			for some positive integers \( m_1, \dots, m_s \) with \( m_1 + \cdots + m_s = n \).
			
			This form is more convenient than the general version of the generalized Jordan normal form. If a block \( J_{m_i}(\alpha_i,1) \) has size \( m_i \ge 2 \), then the desired conclusion follows from a result of Botha. The main difficulty arises when there exists a block of size one, i.e., \( J_{m_i}(\alpha_i,1) = \alpha_i \). In this case, the problem reduces to determining whether \( \alpha_i \) can be expressed as a product of two additive commutators in \( D \). 	This leads to  characterize division rings \( D \) in which every element is a product of two additive commutators.

			Although Theorem~\ref{exist} establishes the existence of such division rings - and, in fact, shows that one can construct examples whose center is any prescribed field - the general structure of these division rings remains unclear.

			These considerations suggest a natural approach to Question~\ref{que}, linking the algebraicity of division rings with the structure theory of matrices over them.

				\subsection{Finitary matrices}
				
				The only remaining case about the field $F$ not yet settled in Theorem~\ref{mainthm} is when the field $F$ has exactly two elements and the matrix size $n$ is even. This particular scenario poses a challenge due to the restriction on  $n$, which prevents the direct use of our earlier techniques. By the way, this also suggests we shift our attention to a broader context, namely the algebra of finitary matrices. In this framework, we can invoke Theorem~\ref{recent} directly to complete the analysis.
				
				For the reader's convenience, we briefly recall the construction of the finitary matrix algebra and defer its proof to a later section. Let $D$ be a division ring. For each positive integer $n$, there is a natural embedding of $\mathrm{M}_n(D)$ into $\mathrm{M}_{n+1}(D)$, defined by mapping a matrix $A$ to $A \oplus (0)$, where the zero matrix is appended as a new row and column. This produces an ascending chain of matrix algebras:
				\[
				\mathrm{M}_1(D) \subseteq \mathrm{M}_2(D) \subseteq \cdots \subseteq \mathrm{M}_n(D) \subseteq \cdots.
				\]
				Taking the union of this chain yields the \emph{finitary matrix algebra} over $D$, denoted by $\mathrm{M}_\infty(D)$:
				\[
				\mathrm{M}_\infty(D) = \bigcup_{n \geq 1} \mathrm{M}_n(D).
				\]
				Each element of $\mathrm{M}_\infty(D)$ is a countably infinite matrix with only finitely many nonzero entries. In particular, every such element can be expressed as $A \oplus 0$ for some finite matrix $A \in \mathrm{M}_n(F)$, where $n \geq 1$ and $0$ denotes the infinite zero matrix. When necessary, we may assume $n \geq 2$ without loss of generality.
				
				With this setting established, we now turn to the following result concerning finitary matrix algebras.
				
				\begin{theorem}\label{finitary1}
					If $D$ is a division ring, then every element in $\mathrm{M}_\infty(D)$ can be expressed as a product of two commutators in $\mathrm{M}_\infty(D)$.
				\end{theorem}
				
				\begin{proof}
					Take any matrix \( A \in \mathrm{M}_\infty(D) \). Then, we can write \( A = A' \oplus 0 \) for some \( A' \in \mathrm{M}_n(D) \) with \( n \geq 2 \). If necessary, we may enlarge \( A' \) to make it singular by adding a zero row and column. Taking Theorem~\ref{recent} into account, \( A' \) can be expressed as a product of two commutators in \( \mathrm{M}_n(D) \), say \( A' = [B, C]\cdot [D,E] \) for some $B,C,D,E\in\mathrm{M}_n(D)$. It follows that \( A = [B \oplus 0, C \oplus 0]\cdot[D \oplus 0, E \oplus 0] \), as required.
				\end{proof}
				
				\begin{remark}
					An alternative proof of Theorem~\ref{finitary1} can be obtained by expressing matrices as products of two nilpotent matrices.
					Let $A \in \mathrm{M}_\infty(D)$ be arbitrary. We may write $A$ in the form $A = A' \oplus 0$ for some $A' \in \mathrm{M}_n(D)$ with $n \ge 2$.
					Next, by adjoining an extra zero row and column, define $A'' = A' \oplus (0)$.
					A straightforward computation shows that
					\[
					A'' =
					\begin{pmatrix}
						0 & A' \\
						0 & 0
					\end{pmatrix}
					\begin{pmatrix}
						0 & \cdots & \cdots & 0 \\
						1 & \ddots &        & \vdots \\
						& \ddots & \ddots & \vdots \\
						&        & 1      & 0
					\end{pmatrix}.
					\]
					By Theorem~\ref{triangular trace zero}, the matrix $A''$ can be written as a product of two commutators in $\mathrm{M}_{n+1}(D)$; that is,
					$A'' = [B, C] \cdot [D, E]$
					for some $B, C, D, E \in \mathrm{M}_{n+1}(D)$.
					Consequently,
					$$
					A = [B \oplus 0,\, C \oplus 0] \cdot [D \oplus 0,\, E \oplus 0],
					$$
					which completes the argument.
				\end{remark}
				
				\section{On the images of decomposable multilinear polynomial maps}
				
				The problem of expressing elements as products of two additive commutators can be placed in a broader perspective. Indeed, it naturally leads to asking whether the image of the multilinear polynomial $f = [x_1, x_2][x_3, x_4]$
				coincides with the entire algebra. This question is closely related to the L’vov--Kaplansky conjecture, which predicts that the image of any multilinear polynomial evaluated on a matrix algebra always forms a vector space.
				
				Motivated by the desire to reduce the complexity of multilinear polynomials whenever possible, it is natural to consider their factorization. This brings us to the following guiding question: if a multilinear polynomial \( p \) is reducible, or more precisely decomposable, does this impose any structure on its image? In this work, we adopt the term \textit{decomposable}, following terminology used in the study of decomposable word maps (see \cite[Page 610]{Pa_Gnu_20}). It is worth noting that the theory of word maps develops in parallel with the study of images of noncommutative polynomials and shares several conceptual similarities. More precisely, we say that a multilinear polynomial is \textit{decomposable} if it can be written as a product of two multilinear polynomials whose sets of variables are disjoint.
				
				One reason for focusing on \textit{products} of multilinear polynomials comes from a simple but useful structural observation: when two multilinear polynomials involve disjoint sets of variables, their product remains multilinear, and its set of variables is exactly the union of the original ones. This elementary remark supports the intuition that the phenomenon predicted by the L’vov-Kaplansky conjecture - that \( p(\mathrm{M}_n(F)) \) is a vector space - may be more common than initially expected. More generally, one can show that \textit{multilinearity is preserved under factorization}: if a multilinear polynomial factors as \( p = p_1 p_2 \cdots p_k \), with none of the factors constant, then each \( p_i \) must itself be multilinear. Conversely, the product of multilinear polynomials with disjoint variable sets is again multilinear.
				
				A natural weakening of the L’vov-Kaplansky conjecture is the so-called Mesyan conjecture \cite[Conjecture 11]{Pa_Me_13}. It states that if \( F \) is a field, \( n \ge 2 \), and \( p \in F\langle x_1, \ldots, x_m \rangle \) is a multilinear polynomial with \( n \ge m-1 \), then the image \( p(\mathrm{M}_n(F)) \) contains all trace-zero matrices in $\mathrm{M}_n(F)$. The term ``weakening'' reflects the fact that a positive solution to the L’vov-Kaplansky conjecture would immediately imply the validity of the Mesyan conjecture. Consequently, any counterexample to the Mesyan conjecture would also disprove the L’vov-Kaplansky conjecture. Furthermore, the paper \cite{Pa_Fa_22} presents formulates a broader version of the Mesyan conjecture (see \cite[Conjecture 3.3]{Pa_Fa_22}). Although the conjecture remains open in general, a related version has been confirmed for finitary matrices over infinite fields; see \cite[Corollary~1.2]{Pa_Vitas_21}.
				
				We now turn to a result of Vitas concerning another version of the Mesyan conjecture in the setting of finitary matrices over an infinite field; see \cite[Corollary~1.2]{Pa_Vitas_21}.
				
				\begin{theorem} \label{Vitas}
					Let \( F \) be an infinite field and let \( p \in F\langle \mathcal{X} \rangle \) be a multilinear polynomial such that \( p(\mathrm{M}_\infty(F)) \neq \{0\} \). Then the image \( p(\mathrm{M}_\infty(F)) \) contains all trace-zero matrices in \( \mathrm{M}_\infty(F) \).
				\end{theorem}
				
				Building on Theorem \ref{Vitas}, we obtain the following theorem.
				
				\begin{theorem} \label{Vitas1}
					Let \( F \) be an infinite field and let  \( p = p_1 p_2 \in F\langle \mathcal{X} \rangle \), where \( p_1, p_2 \in F\langle \mathcal{X} \rangle \) are multilinear polynomials in disjoint sets of variables such that \( p_1(\mathrm{M}_\infty(F)) \neq \{0\} \) and \( p_2(\mathrm{M}_\infty(F)) \neq \{0\} \). Then
					$p(\mathrm{M}_\infty(F)) = \mathrm{M}_\infty(F).$
				\end{theorem}
				
				\begin{proof}
					Let \( A \in \mathrm{M}_\infty(F) \). Then there exists an integer \( n > 1 \) and a matrix \( A' \in \mathrm{M}_n(F) \) such that
					$
					A =
					\begin{pmatrix}
						A' & 0 \\
						0 & 0
					\end{pmatrix}.
					$
					By \cite[Theorem 4.1]{Pa_Botha_97}, we can write \( A' = BC \) for some matrices \( B, C \in \mathrm{M}_n(F) \) with trace zero. It follows that
					$
					\begin{pmatrix}
						B & 0 \\
						0 & 0
					\end{pmatrix}
					$ and $
					\begin{pmatrix}
						C & 0 \\
						0 & 0
					\end{pmatrix}
					$
					are trace-zero matrices in \( \mathrm{M}_\infty(F) \), and hence
					\[
					A =
					\begin{pmatrix}
						A' & 0 \\
						0 & 0
					\end{pmatrix}
					=
					\begin{pmatrix}
						B & 0 \\
						0 & 0
					\end{pmatrix}
					\begin{pmatrix}
						C & 0 \\
						0 & 0
					\end{pmatrix}.
					\]
					By Theorem \ref{Vitas}, we have
					\[
					\begin{pmatrix}
						B & 0 \\
						0 & 0
					\end{pmatrix} \in p_1(\mathrm{M}_\infty(F)),
					\qquad
					\begin{pmatrix}
						C & 0 \\
						0 & 0
					\end{pmatrix} \in p_2(\mathrm{M}_\infty(F)).
					\]
					Therefore, \( A \in p(\mathrm{M}_\infty(F)) \), which completes the proof.
				\end{proof}
				
				In the setting of noncommutative division rings, the situation becomes considerably more delicate. Nevertheless, we are able to obtain results that parallel those established in the commutative case.
				
				\begin{theorem}\label{Vitas2}
					Let \( D \) be a division ring with center \( F \), and let \( p \in F\langle \mathcal{X} \rangle \). Assume that \( F \) is infinite and that \( p = p_1 p_2 \), where \( p_1, p_2 \in F\langle \mathcal{X} \rangle \) are multilinear polynomials in disjoint sets of variables such that \( p_1(\mathrm{M}_\infty(F)) \neq \{0\} \) and \( p_2(\mathrm{M}_\infty(F)) \neq \{0\} \). Then
					\[
					p(\mathrm{M}_\infty(D)) = \mathrm{M}_\infty(D).
					\]
				\end{theorem}
				
				To establish Theorem \ref{Vitas2}, we first require the following auxiliary result.
				
				\begin{lemma}\label{BiSo}
					Let \( D \) be a division ring and let \( n >1  \) be a positive integer. Suppose that \( U \in \mathrm{UT}_n(D) \cup \mathrm{LT}_n(D) \) and \( H = \mathrm{diag}(a_1, \ldots, a_n) \in \mathrm{M}_n(D) \), where \( a_1, \ldots, a_{n-1} \in Z(D) \) and the elements \( a_1, \ldots, a_n \) are pairwise distinct. Then there exist \( P, Q \in \mathrm{GL}_n(D) \) such that
					$
					P^{-1} H U P = H = Q^{-1} U H Q.
					$
				\end{lemma}
				
				\begin{proof}
					It suffices to consider the case \( A = H U \) with \( U \in \mathrm{UT}_n(D) \), and to argue by induction on \( n \). For \( n = 2 \), let
					$
					A =
					\begin{pmatrix}
						a_1 & y \\
						0 & a_2
					\end{pmatrix}.
					$
					Setting \( x = y(a_1 - a_2)^{-1} \), we obtain
					\[
					\begin{pmatrix}
						1 & x \\
						0 & 1
					\end{pmatrix}
					A
					\begin{pmatrix}
						1 & -x \\
						0 & 1
					\end{pmatrix}
					=
					\begin{pmatrix}
						a_1 & 0 \\
						0 & a_2
					\end{pmatrix}.
					\]
					
					Now assume \( n > 2 \) and that the statement holds for matrices of smaller size. We write
					$
					A =
					\begin{pmatrix}
						a_1 & \alpha \\
						0 & A'
					\end{pmatrix},
					$
					where \( A' \in \mathrm{M}_{n-1}(D) \) is upper triangular with diagonal entries \( a_2, \ldots, a_n \), and \( \alpha \) is a row vector.	Let \( A'' = \mathrm{diag}(a_2, \ldots, a_n) \). By the induction hypothesis, there exists \( P \in \mathrm{GL}_{n-1}(D) \) such that \( P^{-1} A' P = A'' \). Note that \( a_1 I_{n-1} - A'' \in \mathrm{GL}_{n-1}(D) \). 	Set \( x' = -\alpha P (a_1 I_{n-1} - A'')^{-1} \) and define
					\[
					P' =
					\begin{pmatrix}
						1 & 0 \\
						0 & P
					\end{pmatrix}
					\begin{pmatrix}
						1 & x' \\
						0 & I_{n-1}
					\end{pmatrix}.
					\]
					A direct computation shows that
					\[
					P'^{-1} =
					\begin{pmatrix}
						1 & -x' P^{-1} \\
						0 & P^{-1}
					\end{pmatrix}
					\quad \text{and} \quad
					P'^{-1} A P' =
					\begin{pmatrix}
						a_1 & 0 \\
						0 & A''
					\end{pmatrix}.
					\]
					This completes the proof.
				\end{proof}
				
				The next result concerns the similarity structure of nilpotent matrices.
				
				\begin{lemma} {\rm\cite[Lemma 3.2]{Pa_AbLe_21}}\label{luy linh}
					Let \( D \) be a division ring and let \( n\geq1 \) be an integer. If \( N \in \mathrm{M}_n(D) \) is a nilpotent matrix, then there exists an invertible matrix \( P \in \mathrm{GL}_n(D) \) such that
					\[
					P^{-1} N P = \bigoplus_{i=1}^s J_{m_i}(0),
					\]
					where \( s, m_1, \ldots, m_s \) are positive integers satisfying \( m_1 + \cdots + m_s = n \). Here,
					\[
					J_{m_i}(0) = \begin{pmatrix}
						0 & 1 & 0 & \cdots & 0 & 0 & 0 \\
						0 & 0 & 1 & \cdots & 0 & 0 & 0 \\
						0 & 0 & 0 & \cdots & 0 & 0 & 0 \\
						\vdots & \vdots & \vdots & \ddots & \vdots & \vdots & \vdots \\
						0 & 0 & 0 & \cdots & 0 & 1 & 0 \\
						0 & 0 & 0 & \cdots & 0 & 0 & 1 \\
						0 & 0 & 0 & \cdots & 0 & 0 & 0
					\end{pmatrix} \in \mathrm{M}_{m_i}(D).
					\]
				\end{lemma}
				
				In other words, every nilpotent matrix over a division ring is similar to a direct sum of Jordan blocks corresponding to the eigenvalue zero, providing a canonical form that will be particularly useful in what follows.
				
				For non-invertible matrices, we rely on the following structural result.
				
				\begin{lemma} {\rm\cite[Theorem 15, Page 28]{Bo_Ja_43}}\label{k kha nghich}
					Let \( D \) be a division ring and let \( n > 1 \) be an integer. Suppose that \( A \in \mathrm{M}_n(D) \setminus \mathrm{GL}_n(D) \) is not nilpotent. Then there exist an invertible matrix \( P \in \mathrm{GL}_n(D) \) and a positive integer \( k < n \) such that
					\[
					P^{-1} A P = G \oplus N,
					\]
					where \( G \in \mathrm{GL}_{n-k}(D) \) and \( N \in \mathrm{M}_k(D) \) is a nilpotent matrix.
				\end{lemma}
				
				This result shows that any non-invertible, non-nilpotent matrix can be decomposed-up to similarity-into an invertible part and a nilpotent part, a structure that will play a key role in our subsequent arguments.
				
				With Proposition \ref{BiSo} in hand, we are ready to prove Theorem \ref{Vitas2}.
				
				\begin{proof}[Proof of Theorem {\rm \ref{Vitas2}}]
					Let \( A \in \mathrm{M}_\infty(D) \) be arbitrary. Then there exists an integer \( n > 1 \) and a matrix \( A' \in \mathrm{M}_n(D) \) such that
					$
					A =
					\begin{pmatrix}
						A' & 0 \\
						0 & 0
					\end{pmatrix}.
					$
					We may assume that \( A' \notin \mathrm{GL}_n(D) \) and that \( n > 2 \). Throughout the proof, we will use the fact that the image is invariant under similarity. First,	if \( A' \) is nilpotent, then the conclusion follows from Lemma \ref{luy linh} and Theorem \ref{Vitas}. Otherwise, assume now that \( A' \) is not nilpotent. By Lemma \ref{k kha nghich}, there exist \( P \in \mathrm{GL}_n(D) \) and a positive integer \( k < n \) such that
					\[
					P^{-1} A' P =
					\begin{pmatrix}
						A_1 & 0 \\
						0 & A_2
					\end{pmatrix},
					\]
					where \( A_1 \in \mathrm{GL}_{n-k}(D) \) and \( A_2 \in \mathrm{M}_k(D) \) is nilpotent.
					
					In what follows,	if \( A_1 \in \{ \lambda \mathrm{I}_{n-k} \mid \lambda \in F \} \), then the conclusion again follows from Lemma~\ref{luy linh} and Theorem \ref{Vitas}. Otherwise, taking Lemma \ref{LHU} into account, there exist matrices \( Q_1 \in \mathrm{GL}_{n-k}(D) \), \( U \in \mathrm{LT}_{n-k}(D) \), \( V \in \mathrm{UT}_{n-k}(D) \), and an element \( h \in D\setminus\{0\} \) such that
					\[
					Q_1^{-1} A_1 Q_1 = U H V,
					\]
					where \( H = \mathrm{diag}(1, \ldots, 1, h) \in \mathrm{GL}_{n-k}(D) \). On the other hand, by Lemma \ref{luy linh}, there exists \( Q_2 \in \mathrm{GL}_k(D) \) such that \( Q_2^{-1} A_2 Q_2 \) is a strictly upper triangular matrix with diagonal entries equal to \( 0 \) and all other entries in \( \{0,1\} \).	Set
					$
					Q = P
					\begin{pmatrix}
						Q_1 & 0 \\
						0 & Q_2
					\end{pmatrix}.
					$
					Then
					\[
					Q^{-1} A' Q =
					\begin{pmatrix}
						U H & 0 \\
						0 & I_k
					\end{pmatrix}
					\begin{pmatrix}
						V & 0 \\
						0 & A_2'
					\end{pmatrix}.
					\]	Since \( F \) is infinite, we can choose \( b_1, \ldots, b_{n-k} \in F \setminus \{0\} \) and \( b_{n-k+1}, \ldots, b_n \in F(h) \setminus \{0\} \) such that
					\[
					b_1 + \cdots + b_{n-k} = 0
					\]
					and
					\[
					b_1^{-1} + \cdots + b_{n-k-1}^{-1} + h b_{n-k}^{-1} + b_{n-k+1}^{-1} + \cdots + b_n^{-1} = 0.
					\] Here, $F(h)$ denotes the field extension of $F$ by adjoining $h$. We note that this extension remains naturally embedded in $D$. Let \( B = \mathrm{diag}(b_1, \ldots, b_n) \). Then both
					$
					\begin{pmatrix}
						U H & 0 \\
						0 & I_k
					\end{pmatrix}
					B^{-1}
					$ and
					$B
					\begin{pmatrix}
						V & 0 \\
						0 & A_2'
					\end{pmatrix}
					$
					have trace zero. Hence,
					\[
					A' =
					Q
					\begin{pmatrix}
						U H & 0 \\
						0 & I_k
					\end{pmatrix}
					B^{-1}
					Q^{-1}
					\cdot
					Q B
					\begin{pmatrix}
						V & 0 \\
						0 & A_2'
					\end{pmatrix}
					Q^{-1}.
					\] By Proposition \ref{BiSo}, the matrices $	Q
					\begin{pmatrix}
						U H & 0 \\
						0 & I_k
					\end{pmatrix}
					B^{-1}
					Q^{-1}$ and $Q B
					\begin{pmatrix}
						V & 0 \\
						0 & A_2'
					\end{pmatrix}
					Q^{-1}$ are similar to diagonal matrices \( B_2 \) and \( B_1 \), respectively. Moreover, since \( p_i(\mathrm{M}_\infty(F)) \neq \{0\} \), it follows that \( p_i(\mathrm{M}_\infty(F(h))) \neq \{0\} \) for \( i = 1,2 \). By Theorem \ref{Vitas},
					\[
					\begin{pmatrix}
						B_i & 0 \\
						0 & 0
					\end{pmatrix}
					\in p_i(\mathrm{M}_\infty(F)) \subseteq p_i(\mathrm{M}_\infty(D)),
					\] so  the sets \( p_1(\mathrm{M}_\infty(D)) \) and \( p_2(\mathrm{M}_\infty(D)) \) contain the corresponding conjugates. Therefore, it implies that $A \in p(\mathrm{M}_\infty(D)),$
					as required.
				\end{proof}

				Continuing along the same line as Theorem \ref{Vitas2}, we arrive at the following result.
				
				\begin{theorem}\label{Waring1}
					Let \( D \) be a division ring with center \( F \), and let \( p \in F\langle \mathcal{X} \rangle \) be a polynomial. Suppose that \( p = p_1 p_2 \), where \( p_1, p_2 \in F\langle \mathcal{X} \rangle \) are multilinear polynomials in disjoint sets of variables such that
					$p_1(\mathrm{M}_2(F)),\; p_2(\mathrm{M}_2(F)) \notin \left\{ \{0\},\; Z(\mathrm{M}_2(F)) \right\}.$
					Then $\mathrm{M}_2(D) = p(\mathrm{M}_2(D)).$
				\end{theorem}

				\begin{proof}
					Let
					$
					A =
					\begin{pmatrix}
						a & b \\
						c & d
					\end{pmatrix}
					\in \mathrm{M}_2(D),
					$
					where \( a, b, c, d \in D \). We consider several cases. 	If \( b = c = 0 \), then \( A \) is diagonal and can be written as
					\[
					A =
					\begin{pmatrix}
						a & 0 \\
						0 & d
					\end{pmatrix}
					=
					\underbrace{
						\begin{pmatrix}
							0 & a \\
							1 & 0
						\end{pmatrix}
					}_{\in \mathrm{M}_2(F(a))}
					\cdot
					\underbrace{
						\begin{pmatrix}
							0 & d \\
							1 & 0
						\end{pmatrix}
					}_{\in \mathrm{M}_2(F(d))}.
					\] We note that $F(a)$ is the field extension of $F$ generated by adjoining $a$. If \( b \neq 0 \), then \( A \) admits the decomposition
					\[
					A =
					\underbrace{
						\begin{pmatrix}
							-1 & 0 \\
							b^{-1}a - db^{-1} & 1
						\end{pmatrix}
					}_{\in \mathrm{M}_2(F(b^{-1}a - db^{-1}))}
					\begin{pmatrix}
						b^{-1} & 0 \\
						b^{-1}a & 1
					\end{pmatrix}^{-1}
					\underbrace{
						\begin{pmatrix}
							0 & -1 \\
							cb - db^{-1}ab & 0
						\end{pmatrix}
					}_{\in \mathrm{M}_2(F(cb - db^{-1}ab))}
					\begin{pmatrix}
						b^{-1} & 0 \\
						b^{-1}a & 1
					\end{pmatrix}.
					\]	If \( c \neq 0 \), then similarly we have
					\[
					A =
					\begin{pmatrix}
						1 & -ac^{-1} \\
						0 & c^{-1}
					\end{pmatrix}^{-1}
					\underbrace{
						\begin{pmatrix}
							0 & ac^{-1}dc - bc \\
							1 & 0
						\end{pmatrix}
					}_{\in \mathrm{M}_2(F(ac^{-1}dc - bc))}
					\begin{pmatrix}
						1 & -ac^{-1} \\
						0 & c^{-1}
					\end{pmatrix}
					\underbrace{
						\begin{pmatrix}
							1 & c^{-1}d - ac^{-1} \\
							0 & -1
						\end{pmatrix}
					}_{\in \mathrm{M}_2(F(c^{-1}d - ac^{-1}))}.
					\]
					
					In each case, the matrix \( A \) can be expressed as a product of matrices lying in matrix algebras over suitable field extensions of \( F \). By \cite[Theorem 1]{Pa_Ma_14}, together with the invariance of polynomial images under similarity, we obtain the desired conclusion.
				\end{proof}
				
				\section{Additive $p$-commutators}\label{sec po}
				
				In an algebra $R$ over a field $F$, additive commutators measure the extent to which multiplication in $R$ fails to be commutative, and they play an important role in many results concerning the images of noncommutative polynomials. More recently, a nonlinear viewpoint-referred to as \textit{polynomial commutators}-has been explored in \cite{Pa_Du_25-1}. From this perspective, one studies expressions of the form $p(ab) - p(ba)$, where $p \in F[x]$ is a polynomial in the variable $x$ with coefficients in a field $F$. When the algebra $R$ is non-unital, we implicitly assume that $p$ has no constant term so that both $p(ab)$ and $p(ba)$ remain well defined. This idea can be traced back to the work of T.~J.~Laffey and T.~T.~West \cite{Pa_LaWe_93}, who considered in the setting of matrix algebras over fields. In this section, we place the results obtained so far in the present paper within the broader framework of polynomial commutators.
				
				As a key technique in \cite{Pa_Du_25-1}, a straightforward substitution shows that if
				\[
				p(x) = \beta_0+\beta_1x +\cdots +\beta_ x^m,
				\]
				where $m \ge 1$ is an integer and $\beta_0, \beta_1,\ldots,\beta_m \in F$, then
				\[
				p[a,b] = \beta_1[a,b] +\beta_2((ab)^2 - (ba)^2) + \cdots +\beta_m((ab)^m - (ba)^m),
				\]
				which is called an {\it additive $p$-commutator}. We notice that every additive $p$-commutator is an additive commutator.
				Indeed, for $a, b\in R$, we have
				\begin{equation}
					p[a,b] =\sum_{k=1}^m\beta_k((ab)^k - (ba)^k)=\sum_{k=1}^m\big[\beta_ka, (ba)^{k-1}b\big]=\left[a, \sum_{k=1}^m\beta_k(ba)^{k-1}b\right],\label{add}
				\end{equation}
				as desired. For the case that $R=\text{M}_n(F)$, where $n>1$ and $F$ is a field of characteristic zero,
				Laffey and West proved that every additive commutator is an additive $p$-commutator (see \cite{Pa_LaWe_93}). Let
				$$
				p[R, R]:=\{p(ab)-p(ba)\mid a, b\in R\}.
				$$
				In this case, $p[R, R]=[R, R]^{[1]}$. In particular, $p[R, R]^+=[R, R]$.

First, we need to clarify the close relationship between polynomials $p(x)\in F[x]$ and $p(xy)-p(yx)\in F\langle x, y \rangle$.
The following is well-known. We refer the reader to the book \cite{Bei_96} for the notion of algebras satisfying a generalized polynomial identity (i.e., GPI).

\begin{lemma}\label{lem5}
Let $R$ be an algebra over a field $F$, and let $L$ be a field extension of $F$. If $F$ is infinite, then R and
$R\otimes_FL$ satisfy the same GPIs with coefficients in $R$.
\end{lemma}

Let $R$ be a prime algebra over a field $F$. Then $F$ is contained in $C$, the extended centroid of $R$. For any $f(x_1,\ldots,x_n)\in F\langle\mathcal{X} \rangle$, $f(R)=0$ if and only if $f(RC)=0$ (see \cite[Theorem~2]{chuang1988}). As a Consequence, $f$ is central-valued on $R$ if and only if $f$ is central-valued on $RC$. These observations allow us to conduct our research in a more general context of prime rings, rather than in the context of prime algebras.

\begin{proposition}\label{pro4}
Let $R$ be a noncommutative prime ring with extended centroid $C$, and let $p \in C[x]$ be a nonconstant polynomial. If
$p(xy)-p(yx)\in C\langle x, y \rangle$ is central-valued on $RC$, then $p(x)$ is also central-valued on $RC$ except when $R\cong {\rm M}_2({\rm GF}(2))$,
					\[
	p(R)^+=\{0, 1, \begin{pmatrix}
						1 & 1 \\
						1 & 0
					\end{pmatrix}, \begin{pmatrix}
						0 & 1 \\
						1 & 1
					\end{pmatrix}\}\ \ \text{\rm and}\ \ p[R, R]^+=\{0, 1\}.
					\]
\end{proposition}

\begin{proof}
Assume that
$p(xy)-p(yx)\in C\langle x, y \rangle$ is central-valued on $R$ but $p(x)$ is not central-valued on $R$.

Then $R$ is a prime PI-ring with $Z(R)\ne \{0\}$ (see \cite[Theorem 2]{rowen1973}).
In this case, its extended centroid $C$ is equal to the quotient field of $Z(R)$. Let $K$ be defined as $C$ if $C$ is a finite field and as the algebraic closure of $C$ if $C$ is an infinite field. Then $$\widehat R:=\text{\rm M}_n(K)\cong RC\otimes_CK,$$ where $\dim_CRC=n^2>1$ since $R$ is not commutative.
It follows from \cite[Theorem 6.4.1]{Bei_96} or \cite[Theorem 2]{chuang1988} that $R$ and $RC$ satisfy the same GPIswith coefficients in $RC$. Applying Lemma \ref{lem5}, we conclude that $p(xy)-p(yx)\in C\langle x, y \rangle$ is central-valued on $\widehat R$.

Let $u, w\in \widehat R$ with $u$ a unit. Then
\begin{equation}\label{eq:5}
p(w)-up(w)u^{-1}=p(w)-p(uwu^{-1})=p(u^{-1}(uw))-p((uw)u^{-1})\in K.
\end{equation}
That is,
$
[p(w), u]u^{-1}\in K.
$
Hence we have
\begin{equation}\label{eq:2}
[p(\widehat R), u]\subseteq Ku
\end{equation}
for any unit $u$ of $\widehat R$. Since $p(x)$ is not central-valued on $RC$, it is clear that $p(x)$ is not central-valued on $\widehat R$.
By Theorem \ref{thm4}, one of the following holds:

Case 1:\ $[\widehat R, \widehat R]\subseteq p(\widehat R)^+$. Choose a noncentral unit $u\in \widehat R$ such that $u\notin [\widehat R, \widehat R]$ (for instance, let  $u=\begin{pmatrix}
						1 & 1 \\
						1 & 0
					\end{pmatrix}$ if $n=2$ and $u=e_{11}+\sum_{j=2}^ne_{jk}$ where $k=n-j+2$ if $n>2$). Then
$\widehat R=[\widehat R, \widehat R]+Ku$. In view of Eq.\eqref{eq:2}, we have
$$
[\widehat R, u]=\big[[\widehat R, \widehat R]+Ku, u\big]=\big[[\widehat R, \widehat R], u\big]\subseteq Ku.
$$
Note that $u^{-1}\widehat R=\widehat R$. Thus
$$
[\widehat R, u]=[u^{-1}\widehat R, u]=u^{-1}[\widehat R, u]\subseteq K.
$$
This implies that $u\in K$, a contradiction.

Case 2:\ $R\cong {\rm M}_2({\rm GF}(2))$ and either
					\[
					p(R)^+=\{0, \begin{pmatrix}
						0 & 1 \\
						1 & 0
					\end{pmatrix}, \begin{pmatrix}
						1& 1 \\
						0 & 1
					\end{pmatrix},  \begin{pmatrix}
						1& 0 \\
						1 & 1
					\end{pmatrix}\} \text{ or }
					p(R)^+=\{0, 1, \begin{pmatrix}
						1 & 1 \\
						1 & 0
					\end{pmatrix}, \begin{pmatrix}
						0 & 1 \\
						1 & 1
					\end{pmatrix}\}.
					\]
In this case, $K=Z(R)=\text{\rm GF}(2)$ and $\widehat R=R$. In view of Eq.\eqref{eq:2}, we have
$$
[p(R), u^2]=\big[[p(R), u], u\big]=0
$$
for all units $u\in R$. If $p(R)^+=\{0, \begin{pmatrix}
						0 & 1 \\
						1 & 0
					\end{pmatrix}, \begin{pmatrix}
						1& 1 \\
						0 & 1
					\end{pmatrix},  \begin{pmatrix}
						1& 0 \\
						1 & 1
					\end{pmatrix}\}$, then
$$
[\begin{pmatrix}
						1& 1 \\
						0 & 1
					\end{pmatrix}, \begin{pmatrix}
						1& 1 \\
						1 & 0
					\end{pmatrix}^2]=\begin{pmatrix}
						1& 1 \\
						0 & 1
					\end{pmatrix}\ne 0,
$$
a contradiction. It follows that $p(R)^+=\{0, 1, \begin{pmatrix}
						1 & 1 \\
						1 & 0
					\end{pmatrix}, \begin{pmatrix}
						0 & 1 \\
						1 & 1
					\end{pmatrix}\}$. Clearly, $p[R, R]^+\subseteq p(R)^+\cap [R, R]\subseteq \{0, 1\}$.

We claim that $1\in p[R, R]^+$ and hence  $p[R, R]^+=\{0, 1\}.$
Otherwise, $p[R, R]^+=\{0\}$. That is, $p(xy)=p(yx)$ for all $x, y\in R$.
In view of Eq.\eqref{eq:5}, $[u, p(w)]=0$ for all $u, w\in R$ with $u$ a unit. Clearly, $R$ is additively generated by all units of $R$. This implies that $p(x)$ is cental-valued on $R$. This is a contradiction since the element $\begin{pmatrix}
						0& 1 \\
						1 & 0
					\end{pmatrix}\in p(R)^+$ is not central.
\end{proof}

\begin{corollary}\label{cor24}
Let $R$ be a prime ring with extended centroid $C$, and let $p \in C[x]$ be a nonconstant polynomial. If $|C|>2$, then
$p(xy)-p(yx)\in C\langle x, y \rangle$ is central-valued on $RC$ if and only if $p(x)$ is central-valued on $RC$.
\end{corollary}
				
				In light of Corollary~\ref{cor21}, we obtain a parallel phenomenon for polynomial commutators.
				
				\begin{theorem}\label{thm52}
Let $R$ be a simple ring with extended centroid $C$, and let $p \in C[x]$ be a nonconstant polynomial. If $p(x)$ is not central-valued on $R$, then one of the following holds:

(i)\ $p[R, R]^+=[R,R]$ except when
$$
R\cong {\rm M}_2({\rm GF}(2))\ \text{\rm and}\ p[R, R]^+=\{0, \begin{pmatrix}
						0 & 1 \\
						1 & 0
					\end{pmatrix}, \begin{pmatrix}
						1& 1 \\
						0 & 1
					\end{pmatrix},  \begin{pmatrix}
						1& 0 \\
						1 & 1
					\end{pmatrix}\}.
$$

(ii)\ $R\cong {\rm M}_2({\rm GF}(2))$,
$p(R)^+=\{0, 1, \begin{pmatrix}
						1 & 1 \\
						1 & 0
					\end{pmatrix}, \begin{pmatrix}
						0 & 1 \\
						1 & 1
					\end{pmatrix}\}$	
and $p[R, R]^+=\{0, 1\}$.
				\end{theorem}
				
				\begin{proof}
Let $f(x,y):=p(xy)-p(yx)\in F\langle x, y\rangle$.

Case 1:\ $p(xy)-p(yx)\in F\langle x, y \rangle$ is not central-valued on $R$. Note from Eq.\eqref{add} that every additive $p$-commutator is an additive commutator. Hence, it follows that $$f(R)^+=p[R, R]^+\subseteq [R,R].$$ On the other hand, since $f(x,y)=p(xy)-p(yx)$ is not central-valued on $R$, Theorem~\ref{thm4} implies that $$[R,R]\subseteq f(R)^+=p[R, R]^+$$ and hence
					$p[R, R]^+=[R,R]$
					except when $R\cong {\rm M}_2({\rm GF}(2))$ and either
$$
p[R, R]^+=\{0, \begin{pmatrix}
						0 & 1 \\
						1 & 0
					\end{pmatrix}, \begin{pmatrix}
						1& 1 \\
						0 & 1
					\end{pmatrix},  \begin{pmatrix}
						1& 0 \\
						1 & 1
					\end{pmatrix}\}\ \text{and}\
p[R, R]^+=\{0, 1, \begin{pmatrix}
						1 & 1 \\
						1 & 0
					\end{pmatrix}, \begin{pmatrix}
						0 & 1 \\
						1 & 1
					\end{pmatrix}\}.
$$
A direct inspection shows that the latter set is not contained in $[R,R]$, for instance by comparing the traces of its elements. Hence this possibility cannot occur, and we are left with the first exceptional case.

Case 2:\ $p(xy)-p(yx)\in F\langle x, y \rangle$ is central-valued on $R$. It follows from Proposition \ref{pro4} that $R\cong {\rm M}_2({\rm GF}(2))$,
$p(R)^+=\{0, 1, \begin{pmatrix}
						1 & 1 \\
						1 & 0
					\end{pmatrix}, \begin{pmatrix}
						0 & 1 \\
						1 & 1
					\end{pmatrix}\}$ and $p[R, R]^+=\{0, 1\}$.
\end{proof}

\begin{corollary}\label{cor25}
Let $R$ be a simple ring with extended centroid $C$, and let $p \in C[x]$ be a nonconstant polynomial, which is not central-valued on $R$.
If $|C|>2$, then $p[R, R]^+=[R,R].$
\end{corollary}
				
				Theorem~\ref{thm52} naturally leads to the following.
				
				\begin{lemma}\label{lem71}
				Let $R$ be a noncommutative prime ring with extended centroid $C$, and let $p \in C[x]$ be a nonconstant polynomial of degree $n>1$. If $C$ contains at least $n+1$ distinct elements, then $p(x)$  is not central-valued on $RC$.
				\end{lemma}
				
				\begin{proof}
Suppose, to the contrary, that $p(x)$  is central-valued on $RC$. Then $R$ must be a PI-algebra with $Z(R)\ne \{0\}$ (see  \cite[Theorem 2]{rowen1973}).
Thus $C$ contains at least $n+1$ distinct elements. Write $\displaystyle p(x)=\sum_{i=0}^n\beta_ix^i$, where $\beta_i\in F$ and $\beta_n\ne 0$.
					Then
					$$
					p(\lambda a) - \beta_0 =\sum_{i=1}^n\lambda^i\beta_ia^i\in C
					$$
					for all $a\in R$ and $\lambda\in C$.  Solving it by a Vandermonde matrix argument, we get $a^n\in C$ for all $a\in R$ (see \cite[Lemma 2.2]{bresar2009}). Then $x^ny=yx^n$ for all $x, y\in R$. By \cite[Theorem, p.19]{lee1996},  it follows that $[x, y]=0$ for all $x, y\in R$.
This implies the commutativity of $R$, a contradiction.
				\end{proof}
				
				\begin{theorem}\label{thm73}
					Let $R$ be a noncommutative prime algebra over a field $F$, and let $p \in F[x]$ be a nonconstant polynomial of degree $n\geq 1$.
					Assume that  either $Z(R)=\{0\}$ or $Z(R)$ contains at least $n+1$ distinct elements.
					Then
					$[M, R]\subseteq p[R,R]^+\subseteq [R, R]$ for some nonzero ideal $M$ of $R$. In addition, if $R$ is a simple ring, then
					$p[R,R]^+=[R, R]$.
				\end{theorem}
				
				\begin{proof}
 As before, we always have $p[R,R]^+\subseteq [R, R]$.
					Clearly, we may assume that $p$ has no constant term.
					For $n=1$, it is clear that $p[R,R]^+=[R, R]$. Thus we assume $n\geq 2$.

					We claim that $p(x)$  is not central-valued on $R$. Otherwise, $p(x)$  is central-valued on $R$ and hence $R$ is a prime PI-algebra.
Thus $Z(R)\ne 0$ (see \cite[Theorem 2 and Corollary 1]{rowen1973}). Note that $C$, the extended centroid of $R$, is the quotient field of $Z(R)$. Hence $|C|\geq n+1$.
By Lemma \ref{lem71}, $p(x)$  is not central-valued on $R$, a contradiction. This proves the claim.

Hence $p(xy)-p(yx)\in F\langle x, y \rangle$ is not central-valued on $R$ (see Proposition \ref{pro4}).
In view of Theorem \ref{thm4}, $[M, R]\subseteq p[R,R]^+$ for some nonzero ideal $M$ of $R$ since either $Z(R)=\{0\}$ or $Z(R)$ contains at least $n+1$ distinct elements. In addition, if $R$ is a simple ring, then $M=R$ and so
					$p[R,R]^+=[R, R]$. This completes the proof.
				\end{proof}				
				
				Herstein also obtained a closely related result concerning $\overline{[R,R]}$. He proved that if $R$ is a noncommutative simple ring, then $R=\overline{[R,R]}$ (see \cite[Corollary, p.~6]{Bo_Her_69} together with the remark immediately following it).  Combining this classical observation with Theorem~\ref{thm73}, we obtain the following.
				
				\begin{corollary}\label{cor23}
					If $R$ is a noncommutative simple $F$-algebra and  let $p \in F[x]$ be a nonconstant
					polynomial, then $p[R,R]^+=[R, R]$  and ${\overline {p[R,R]}}=R$ except when $|R|<\infty$.
				\end{corollary}
				
				\begin{proof}
					Assume that $|R|=\infty$. We claim that $p(xy)-p(yx)\in F\langle x, y \rangle$ is not central-valued on $R$.
					Otherwise, $R$ is a PI-algebra. In this case, $Z(R)\ne 0$ and $\dim_{Z(R)}R<\infty$ (see \cite[Theorem 2 and Corollary 1]{rowen1973}). Hence $Z(R)$ is an infinite field. In view of Lemma \ref{lem71}, $p(x)$  is not central-valued on $R$.
 By Proposition \ref{pro4}, the polynomial $p(xy)-p(yx)\in F\langle x, y \rangle$ is not central-valued on $R$, a contradiction.	Thus the claim is proved.
It follows from Theorem \ref{thm73} that $p[R,R]^+=[R, R]$ and hence, by Herstein's theorem,  ${\overline {p[R,R]}}=R$.
				\end{proof}

				\section{The phenomenon: $\mathrm{M}_n(D)=\mathrm{SL}_n(D)-\mathrm{SL}_n(D)$}\label{sec phe}
				
				A key starting point is \cite[Theorem~4.9]{Pa_PaSo_26}, which shows that every matrix over a centrally finite algebraically closed division ring can, under suitable assumptions, be expressed as a difference of two multiplicative commutators from images of polynomials.  A crucial ingredient in the proof is \cite[Lemma~4.7]{Pa_PaSo_26}, asserting that if $D$ is a division ring and $n > 1$, then for every matrix $A \in \mathrm{M}_n(D)$ there exist matrices $B, C \in \mathrm{SL}_n(D)$ such that $A = B - C$. This lemma, in fact, goes back to M.~Mahdavi-Hezavehi \cite[Lemma~5.1]{Pa_Mad_00}, who established the striking identity
				$\mathrm{M}_n(D) = \mathrm{SL}_n(D) - \mathrm{SL}_n(D)$
				for every integer $n \ge 2$. Beyond its intrinsic interest, this phenomenon has significant structural consequences. Building on it, Mahdavi-Hezavehi proved that any subgroup $G$ of finite index in $\mathrm{GL}_n(D)$ is contained in a maximal normal subgroup of finite index in $\mathrm{GL}_n(D)$. As an application, he answered two questions posed in \cite{Pa_Be_92} concerning $(G - G)$-rings over a division algebra of finite dimension over its centre. In light of this result, it is natural to ask whether the same equality remains valid in the case $n = 1$. Somewhat surprisingly, the answer is negative - even in one of the most classical noncommutative settings, namely the real quaternion division ring.

				To set the stage, we briefly recall some notation. As a set, the real quaternion division ring is $\mathbb{H}
				= \{ a + bi + cj + dk \mid a,b,c,d \in \mathbb{R},\ i^2 = j^2 = k^2 = ijk = -1 \}.$
				The structure of the set generated by multiplicative commutators in $\mathbb{H}$ was already investigated by Wang in 1950 \cite{Pa_Wa_50}. He proved that the subgroup  $\mathrm{SL}_1(\mathbb{H})$ generated by all multiplicative commutators in $\mathbb{H}$ coincides with
				\[
				\mathrm{SL}_1(\mathbb{H})
				= \{ a + bi + cj + dk \mid a,b,c,d \in \mathbb{R},\ a^2+b^2+c^2+d^2 = 1 \}.
				\]
				In view of Wang's result, the problem we are led to consider is the following: which quaternions can be expressed as the difference of two elements of norm one?
				
				In what follows, for each quaternion $\alpha = a + bi + cj + dk \in \mathbb{H}$, we  define its conjugate by
				$\overline{\alpha} = a - bi - cj - dk,$
				its norm by
				$N(\alpha) = a^2 + b^2 + c^2 + d^2,$
				and its length by $\|\alpha\| = \sqrt{N(\alpha)}.$
				If $v = v_0 + v_1 i + v_2 j + v_3 k$ and $q = q_0 + q_1 i + q_2 j + q_3 k$ are elements of $\mathbb{H}$, we denote by
				$\langle v, q \rangle = v_0 q_0 + v_1 q_1 + v_2 q_2 + v_3 q_3$
				their Euclidean inner product. With this notation, $\|q\| = \sqrt{\langle q,q \rangle}$.
				
				\begin{remark}\label{uv}
					For all $u,v \in \mathbb{H}$ one has the inequality $\|u - v\| \le \|u\| + \|v\|$.
					Indeed, note first that for any $\alpha,\beta \in \mathbb{H}$ we have $N(\alpha) = \|\alpha\|^2$ and
					$\operatorname{Re}(\alpha \overline{\beta}) = \langle \alpha,\beta \rangle$. Hence,
					$$\|u - v\|^2 = N(u - v)
					= (u - v)(\overline{u} - \overline{v})
					= N(u) + N(v) - u\overline{v} - v\overline{u}.$$
					Taking real parts and using $\operatorname{Re}(v\overline{u}) = \operatorname{Re}(u\overline{v})$, we obtain
					$$\|u - v\|^2 = N(u) + N(v) - 2\,\operatorname{Re}(u\overline{v}).$$
					Applying the Cauchy--Schwarz inequality in $\mathbb{R}^4$ (or the estimate
					$|\operatorname{Re}(z)| \le \|z\|$ for all $z \in \mathbb{H}$), we get that $|\operatorname{Re}(u\overline{v})| \le \|u\overline{v}\| = \|u\|\,\|v\|.$
					Hence, it follows that
					$$\|u - v\|^2 \le N(u) + N(v) + 2\|u\|\,\|v\|
					= (\|u\| + \|v\|)^2.$$
					Since both sides are nonnegative, the desired inequality $\|u - v\| \le \|u\| + \|v\|$ follows.
				\end{remark}
				
				We are now in a position to address the  question: which quaternions can be written as the difference of two elements of norm one? The following theorem provides a complete answer to this question.
				
				\begin{theorem}\label{qu}
					A quaternion $q \in \mathbb{H}$ can be written in the form $q = u - v$, where
					$u, v \in \mathbb{H}$ satisfy $N(u) = N(v) = 1$, if and only if $\|q\| \le 2$.
				\end{theorem}

				\begin{proof}
					Let $q \in \mathbb{H}$ be arbitrary. Suppose first that $q = u - v$ for some
					$u, v \in \mathbb{H}$ with $N(u) = N(v) = 1$. By Remark~\ref{uv}, we obtain $\|q\| = \|u - v\| \le \|u\| + \|v\| = 2$, as promised. We now prove the converse. Assume that $\|q\| \le 2$. If $\|q\| = 0$, then
					$q = 0 = 1 - 1$, so we may take $u = v = 1$. Henceforth, assume that
					$0 < \|q\| \le 2$.
					
					Our goal is to find $v \in \mathbb{H}$ with $N(v) = 1$ such that, setting
					$u := v + q$, we also have $N(u) = 1$. Equivalently, we seek $v \in \mathbb{H}$
					satisfying $\|v\| = 1$ and $\|v + q\| = 1$. This is the same as requiring $\|v + q\|^2 - \|v\|^2 = 0.$
					Expanding the left-hand side yields
					$$\|v + q\|^2
					= \langle v + q, v + q \rangle
					= \|v\|^2 + 2\langle v, q \rangle + \|q\|^2,$$
					and therefore $2\langle v, q \rangle + \|q\|^2 = 0,$ or equivalently
					$\langle v, q \rangle = -\frac{\|q\|^2}{2}.$ Let $e = q / \|q\|$. Then
					$$\langle v, e \rangle
					= \frac{1}{\|q\|}\langle v, q \rangle
					= -\frac{\|q\|}{2}.$$
					Thus, it suffices to find a unit quaternion $v$ such that its inner product
					with $e$ equals $-\|q\|/2$. Since $\|q\| \le 2$, we have $-1 \le -\frac{\|q\|}{2} \le 0.$ Choose a unit quaternion $w \in \mathbb{H}$ orthogonal to $e$, that is,
					$\|w\| = 1$ and $\langle w, e \rangle = 0$. Define
					$v = -\frac{\|q\|}{2}\, e
					+ \sqrt{1 - \frac{\|q\|^2}{4}}\, w.$
					Using the standard formula for the squared norm of a sum, together with the
					orthogonality of $e$ and $w$, we can compute $\|v\|^2
					= \frac{\|q\|^2}{4}
					+ \left(1 - \frac{\|q\|^2}{4}\right)
					= 1,$
					and
					$\langle v, e \rangle = -\frac{\|q\|}{2}.$
					Hence $v$ has unit norm and satisfies the required inner product condition.
					Finally, setting $u := v + q$, we obtain $\|u\| = 1$ and $u - v = q$, which
					completes the proof.
				\end{proof}
				
				Besides, Theorem~\ref{qu}, combined with the classical result of Wang \cite{Pa_Wa_50}, yields the identity
				$\mathrm{SL}_1(\mathbb{H}) - \mathrm{SL}_1(\mathbb{H})
				= \{\, q \in \mathbb{H} \mid \|q\| \le 2 \,\},$
				which is a proper subset of $\mathbb{H}$. On the other hand, it is shown in \cite[Lemma 2.1]{Pa_Du_25} that every element of $\mathrm{SL}_1(\mathbb{H})$ can be expressed as a single commutator of the form $aba^{-1}b^{-1}$, where $a, b \in \mathbb{H}$ satisfy $a^2 = b^2 = -1$. Moreover, one easily verifies that $$\{\, a \in \mathbb{H} \mid a^2 = -1 \,\}
				=
				\{\, \alpha i + \beta j + \gamma k \mid \alpha, \beta, \gamma \in \mathbb{R},\; \alpha ^2 + \beta^2 + \gamma^2 = 1 \,\}.$$ Furthermore, by \cite[Theorem 1]{Pa_Ma_21},  if $p \in \mathbb{R}\langle \mathcal{X} \rangle$ is  multilinear  such that $p(\mathbb{H}) \notin \{\{0\}, \mathbb{R}\}$, then the entire space of purely imaginary quaternions
				$\{\, \alpha i + \beta j + \gamma k \mid \alpha, \beta, \gamma \in \mathbb{R} \,\}$
				is contained in $p(\mathbb{H})$. Note that a multilinear polynomial has the form
				$$\displaystyle p = \sum_{\sigma \in S_m} \lambda_\sigma \, x_{\sigma(1)} x_{\sigma(2)} \cdots x_{\sigma(m)},$$
				where $S_m$ denotes the symmetric group of degree $m$, and each coefficient $\lambda_\sigma$ belongs to $\mathbb{R}$. As a consequence, we obtain that $\{\, a \in \mathbb{H} \mid a^2 = -1 \,\} \subseteq p(\mathbb{H}).$ Combining these observations, we arrive at the following theorem.
				
				\begin{theorem}\label{qq}
					\
					\begin{enumerate}[\rm (i)]
						\item If $p_1, p_2 \in \mathbb{R}\langle \mathcal{X} \rangle$ are multilinear polynomials such that $p_1(\mathbb{H}), p_2(\mathbb{H}) \notin \{\{0\}, \mathbb{R}\}$, then
						$\mathrm{SL}_1(\mathbb{H})
						\subseteq
						\{\, xyx^{-1}y^{-1} \mid x \in p_1(\mathbb{H}),\; y \in p_2(\mathbb{H}) \,\}.$
						\item If $q \in \mathbb{H}$ and  $p_i \in \mathbb{R}\langle \mathcal{X} \rangle$, for $i \in \{1,2,3,4\}$, are multilinear polynomials such that $\|q\| \le 2$ and $p_i(\mathbb{H}) \notin \{\{0\}, \mathbb{R}\}$ for each $i$, then there exist elements $a_i \in p_i(\mathbb{H})$, for $i \in \{1,2,3,4\}$, such that
						$q
						=
						a_1 a_2 a_1^{-1} a_2^{-1}
						-
						a_3 a_4 a_3^{-1} a_4^{-1}.$
					\end{enumerate}
				\end{theorem}
				
				Theorem~\ref{qq} clarifies the situation in the case $n = 1$, when the underlying division ring is the real quaternion division ring. It is therefore natural to ask how the picture changes once we pass to higher matrix sizes, that is, when $n \geq 2$. The following theorem shows that, in contrast to the one-dimensional case, a much stronger structural phenomenon occurs.
				
				\begin{theorem}\label{real}
					Let $n \geq 2$ be an integer.
					\begin{enumerate}[\rm (i)]
						\item If $p_i \in \mathbb{R}\langle \mathcal{X} \rangle$, for $i \in \{1,\dots,4\}$, are multilinear polynomials with $p_i(\mathbb{H}) \notin \{\{0\}, \mathbb{R}\}$ for each $i$, then
						$$\mathrm{SL}_n(\mathbb{H}) \subseteq
						\left\{
						A_1A_2A_1^{-1}A_2^{-1}
						A_3A_4A_3^{-1}A_4^{-1}
						\;\middle|\;
						A_i \in p_i\big(\mathrm{M}_n(\mathbb{H})\big)
						\right\}.$$
						\item If $p_i \in \mathbb{R}\langle \mathcal{X} \rangle$, for $i \in \{1,\dots,8\}$, are multilinear polynomials with $p_i(\mathbb{H}) \notin \{\{0\}, \mathbb{R}\}$ for each $i$, then every matrix $A$ in $	\mathrm{M}_n(\mathbb{H})$ can be expressed as $$A=A_1A_2A_1^{-1}A_2^{-1}
						A_3A_4A_3^{-1}A_4^{-1}
						-
						A_5A_6A_5^{-1}A_6^{-1}
						A_7A_8A_7^{-1}A_8^{-1}$$ for some $A_i \in p_i\big(\mathrm{M}_n(\mathbb{H})$.
					\end{enumerate}
				\end{theorem}
				
				In particular, when $n \geq 2$, every element of $\mathrm{SL}_n(\mathbb{H})$, and even every matrix in $\mathrm{M}_n(\mathbb{H})$, can be described in terms of products and differences of multiplicative commutators whose entries lie in suitable polynomial images. This highlights a striking contrast with the case $n = 1$ and underscores the richer structure that emerges in higher dimensions.
				
				Our strategy for proving Theorem~\ref{real} is based on the following result, which concerns products of two commutators of skew involutions. Note that a skew involution in a ring is an element $a$ such that $a^2 = -1$.
				After establishing this structural statement, we then exploit the similarity of skew involutions to complete the argument.

				\begin{lemma}\label{skew}
					If $n\geq2$ is an integer, then every element of
					$\mathrm{SL}_n(\mathbb H)$ can be expressed as a product of at most two multiplicative
					commutators of skew involutions in $\mathrm{M}_n(\mathbb H)$.
				\end{lemma}
				
				\begin{proof}
					Let $A \in \mathrm{SL}_n(\mathbb H)$. Taking \cite[Lemma~16]{Pa_Ha_Nam_Son} into account, the center of
					$\mathrm{SL}_n(\mathbb H)$ is given by
					$Z(\mathrm{SL}_n(\mathbb H))=\{\pm \mathrm I_n\}.$
					We therefore distinguish two cases.
					
					\medskip
					\noindent
					\textit{Case~1: $A$ is central.}
					If $A=\mathrm I_n=(i\mathrm I_n)(i\mathrm I_n)(i\mathrm I_n)^{-1}(i\mathrm I_n)^{-1},$
					where $(i\mathrm I_n)^2=-\mathrm I_n$, then $A$ is a multiplicative commutator of skew involutions in
					$\mathrm{M}_n(\mathbb H)$. Moreover,
					if $A=-\mathrm I_n$, then it is not difficult to verify that
					$-\mathrm I_n
					=iji^{-1}j^{-1}\mathrm I_n
					=(i\mathrm I_n)(j\mathrm I_n)(i\mathrm I_n)^{-1}(j\mathrm I_n)^{-1},$
					and both $i\mathrm I_n$ and $j\mathrm I_n$ are skew involutions. Thus the conclusion
					holds in the central case.
					
					\medskip
					\noindent
					\textit{Case~2: $A$ is noncentral.}
					We consider two subcases depending on the parity of $n$.
					
					Suppose first that $n=2s$ with $s\ge1$. By Lemma~\ref{LHU}, we may
					assume that $A$ admits a factorization $A=XTY$, where $X$ is unit lower triangular,
					$Y$ is unit upper triangular, and
					$T=\mathrm{diag}(1,\ldots,1,t)$ with $t\in\mathrm{SL}_1(\mathbb{H})$.
					By \cite[Lemma~2.5]{Pa_BiDuHa}, the element $t$ can be written as a multiplicative
					commutator $t=xyx^{-1}y^{-1}$ for some $x,y\in\mathbb{H}$.
					Choosing pairwise distinct elements $h_1,\ldots,h_{s-1}$ from
					$\mathbb R\setminus\{0,\pm1\}$, we obtain a decomposition
					$T=UVU^{-1}V^{-1}$, where $U$ and $V$ are diagonal matrices of the prescribed form: $$ U = \operatorname{diag}(\underbrace{1,\ldots,1}_{n-1\ \text{times}},\, x); V = \operatorname{diag}\bigl(
					h_1^{2}, h_1^{-2},
					\ldots,
					h_{s-1}^{2}, h_{s-1}^{-2},
					y^{-1}, y
					\bigr).$$
					
					Rewriting $A = XUVU^{-1}V^{-1}Y = U\big(U^{-1}XUV\big)U^{-1} V^{-1}Y$ accordingly, we observe that the matrix $U^{-1}XUV$ is lower triangular
					with pairwise nonconjugate diagonal entries. By \cite[Lemma~3.2]{Pa_BiDuHaSo_2022},
					it is therefore similar to a diagonal matrix whose entries are
					$h_1^2,h_1^{-2},\ldots,y^{-1},y$.
			By \cite[Lemma 14]{Pa_Ha_Nam_Son}, there exists $k\in\mathbb H$ such that $y=k^2$.
					A direct computation then shows that this diagonal matrix is itself a commutator of
					two skew involutions in $\mathrm{M}_n(\mathbb H)$, that is, $$ \operatorname{diag}(h_1^2, h_1^{-2}, \ldots, k^{-2}, k^2)= SKS^{-1}K^{-1},$$ where
					\[
					S=\bigoplus_{i=1}^{s-1}
					\begin{pmatrix}
						0 & h_i\\
						- h_i^{-1} & 0
					\end{pmatrix}
					\;\oplus\;
					\begin{pmatrix}
						0 & k^{-1}\\
						- k & 0
					\end{pmatrix},
					\qquad
					K=\bigoplus_{i=1}^{s}
					\begin{pmatrix}
						0 & 1\\
						-1 & 0
					\end{pmatrix}.
					\]
					An analogous argument applies to the matrix $V^{-1}Y$.
					Hence $A$ is a product of two multiplicative commutators of skew involutions when $n$ is even.
					
					Now suppose that $n=2s+1$ with $s\ge1$. Using the same reduction as above, we write
					$A=XTY$, where $T=\mathrm{diag}(1,\ldots,1,t)$ and $t$ is a multiplicative commutator in $\mathbb H$.
					The matrix $T$ decomposes as a direct sum of $1$ and an element of
					$\mathrm{SL}_{2s}(\mathbb H)$, that is, $T =(1) \oplus \mathrm{diag}(1,1,\ldots,1,t),$ with  $\mathrm{diag}(1,1,\ldots,1,t)\in \mathrm{SL}_{2s}(\mathbb{H})$. This allows us to express $T$ as a multiplicative commutator by
					embedding the even-dimensional construction into $\mathrm{M}_n(\mathbb H)$. Indeed, \begin{eqnarray*}
						T &=& (1) \oplus UVU^{-1}V^{-1} \\& =& [(-1)\oplus U][(-1) \oplus V][(-1) \oplus U^{-1}][(-1)\oplus V^{-1}] = U'V'U'^{-1}V'^{-1}.
					\end{eqnarray*}
					Consequently, $A =   U' \big(U'^{-1}X U'V'\big)U'^{-1}V'^{-1}Y$ can again be written as a product of two multiplicative commutators. Indeed, since $-1=iji^{-1}j^{-1}$ itself is a multiplicative commutator of skew involutions in $\mathbb{H}$, the similarity arguments carry
					over without change, that is, the matrix $(-1) \oplus  \mathrm{diag}(h_1^2,h_1^{-2},\ldots,k^{-2},k^2) $ is a multiplicative commutator of skew involutions in $\mathrm{M}_n(\mathbb{H})$. Hence, $U'^{-1} XU'V'$  can be expressed as a multiplicative commutator of skew involution, so is $V'^{-1}Y$. This completes the proof.
				\end{proof}
				
				As mentioned above, the subsequent result below describes skew involutions up to similarity. Its proof may be obtained by a minor modification of \cite{Pa_BiDuHaSo_2023}; nevertheless, for clarity and completeness, we provide a short proof here.
				
				\begin{lemma}\label{skew1}
					Let $D$ be a division ring and let $n\geq2$ be an integer. In the case $\mathrm{char}(D)\neq2$, we additionally assume that $D$ contains a skew involution $\alpha$. If $A\in\mathrm{GL}_n(D)$ is a skew involution, then there exists a non-negative integer $r$ such that $A$ is similar to $$\begin{cases}
						\operatorname{diag}(\alpha \mathrm{I}_r,-\alpha \mathrm{I}_{n-r}) \text{ if } \operatorname{char}(D)\neq2,\\
						\begin{pmatrix}
							\mathrm{I}_r & \mathrm{I}_r\\
							0 & \mathrm{I}_r
						\end{pmatrix}\oplus \mathrm{I}_{n-2r}  \text{ if } \operatorname{char}(D)=2.
					\end{cases}$$
				\end{lemma}
				
				\begin{proof}
					We distinguish two cases. First, we assume that $\operatorname{char}(D)=2$. In this situation, it is not difficult to verify that $A^{2}=-\mathrm{I}_{n}=\mathrm{I}_{n},$
					which means that $A$ is an involution. Therefore, by applying \cite[Proposition~2.3]{Pa_BiDuHaSo_2023}, we conclude that $A$ must have the required form.
					
					Now assume that $\operatorname{char}(D)\neq 2$. Set $B=A+\alpha \mathrm{I}_{n}$ and  $C=A-\alpha \mathrm{I}_{n}.$
					A straightforward computation shows that $BC=CB=A^{2}-\alpha^{2}\mathrm{I}_{n}=0.$ Taking \cite[Lemma~2.1]{Pa_BiDuHaSo_2023} into account, there exist a non-negative integer $r$ and an invertible matrix
					$T\in \mathrm{GL}_{n}(D)$ such that $$TBT^{-1}=
					\begin{pmatrix}
						B_{11} & B_{12} \\
						0 & 0
					\end{pmatrix},$$
					where $B_{11}\in \mathrm{M}_{r}(D)$ and $B_{12}\in \mathrm{M}_{r\times (n-r)}(D)$. For simplicity, denote $B_{1}=TBT^{-1}$ and $C_{1}=TCT^{-1}.$ From the relation $B-C=2\alpha \mathrm{I}_{n}$, we obtain $B_{1}-C_{1}=2\alpha \mathrm{I}_{n}$, and hence
					$$C_{1}=
					\begin{pmatrix}
						B_{11}-2\alpha \mathrm{I}_{r} & B_{12} \\
						0 & -2\alpha \mathrm{I}_{n-r}
					\end{pmatrix}.$$ Since $C_{1}B_{1}=0$, it follows that $(B_{11}-2\alpha \mathrm{I}_{r})(\,B_{11}\;\; B_{12}\,)=0.$
					Moreover, \cite[Lemma~2.2]{Pa_BiDuHaSo_2023} ensures the existence of
					$Q\in \mathrm{GL}_{n}(D)$ such that
					$(\,B_{11}\;\; B_{12}\,)Q=(\mathrm{I}_{r}\;\; 0).$
					Consequently, we must have $B_{11}-2\alpha \mathrm{I}_{r}=0$, that is, $B_{11}=2\alpha \mathrm{I}_{r}.$ Next, define $P=(2\alpha)^{-1}B_{12}$ and consider the invertible matrix $U=
					\begin{pmatrix}
						\mathrm{I}_{r} & P \\
						0 & \mathrm{I}_{n-r}
					\end{pmatrix}.$
					A direct block computation then shows that
					$$U(TAT^{-1})U^{-1}=
					\begin{pmatrix}
						\alpha \mathrm{I}_{r} & 0 \\
						0 & -\alpha \mathrm{I}_{n-r}
					\end{pmatrix},$$
					as promised.
				\end{proof}
				
				With Lemmas~\ref{skew} and~\ref{skew1} at our disposal, we are now in a position to complete the proof of Theorem~\ref{real}.
				
				\begin{proof}[Proof of Theorem~\ref{real}]
					Let $A \in \mathrm{SL}_n(\mathbb{H})$. By Lemma~\ref{skew}, we may write
					\[
					A = A_1A_2A_1^{-1}A_2^{-1}A_3A_4A_3^{-1}A_4^{-1},
					\]
					where each $A_i \in \mathrm{M}_n(\mathbb{H})$ satisfies $A_i^2 = -\mathrm{I}_n$.	In view of Lemma~\ref{skew1}, every such $A_i$ is similar to a matrix of the form
					$\mathrm{diag}(i\mathrm{I}_r, -i\mathrm{I}_{n-r})$
					for some nonnegative integer $r$. On the other hand, by \cite[Theorem 1]{Pa_Ma_21}, if $p \in \mathbb{R}\langle \mathcal{X} \rangle$ is a multilinear polynomial with $p(\mathbb{H}) \notin \{\{0\}, \mathbb{R}\}$, then both $i$ and $-i$ lie in $p(\mathbb{H})$. It follows readily that the diagonal matrix
					$\mathrm{diag}(i\mathrm{I}_r, -i\mathrm{I}_{n-r})$
					belongs to $p(\mathrm{M}_n(\mathbb{H}))$. Since the set $p(\mathrm{M}_n(\mathbb{H}))$ is invariant under similarity, we conclude that each $A_i$ lies in $p(\mathrm{M}_n(\mathbb{H}))$, which establishes part~(i).	Part~(ii) follows immediately from~(i) together with the result of M.~Mahdavi-Hezavehi \cite[Lemma~5.1]{Pa_Mad_00}, who proved the remarkable identity
					$\mathrm{M}_n(D) = \mathrm{SL}_n(D) - \mathrm{SL}_n(D)$
					for every integer $n \ge 2$.	This completes the proof.
				\end{proof}

				\bigskip
				
				\section*{Declarations}
				
				Our statements here are the following:
				
				\begin{itemize}
					\item {\bf Ethical Declarations and Approval:} The authors have no any competing interest to declare that are relevant to the content of this article.
					\item {\bf Competing Interests:} The authors declare no any conflict of interest.
					\item  {\bf Authors' Contributions:} All two listed authors worked and contributed to the paper equally. The final editing was done by the corresponding author Tran Nam Son and was approved by all of the present authors.
					\item {\bf Availability of Data and Materials:} Data sharing not applicable to this article as no data-sets or any other materials were generated or analyzed during the current study.
				\end{itemize}

				\bigskip
				
				\bibliographystyle{amsplain}

				\vskip 0.2cm

				\noindent Tsiu-Kwen Lee\\
				Department of Mathematics, National Taiwan University, Taipei, Taiwan\\
				Email: tklee@math.ntu.edu.tw,\\ ORCID: \href{https://orcid.org/0000-0002-1262-1491}{0000-0002-1262-1491}

				\vskip 0.2cm
				
				\noindent Tran Nam Son\\ Department of Mathematics, Dong Nai University, 9 Le Quy Don Str., Tam Hiep Ward, Dong Nai City, Vietnam\\ Email: trannamson1999@gmail.com or sontn@dnpu.edu.vn,\\ ORCID: \href{https://orcid.org/0000-0002-9560-6392}{0000-0002-9560-6392}

			\end{document}